\documentclass[a4paper,12pt]{article}
\usepackage{amsfonts,amsmath,amssymb,hhline,amsthm,cite,latexsym,amscd}
\usepackage[T2A]{fontenc}
\usepackage[cp1251]{inputenc}
\usepackage[english]{babel}
\usepackage{graphicx}
\usepackage{tikz}
\usetikzlibrary{calc,intersections,through,math,quotes,angles,babel}

\theoremstyle{definition}
\newtheorem{definition}{Definition}[section]
\newtheorem{remark}[definition]{Remark}
\newtheorem{example}[definition]{Example}

\theoremstyle{theorem}
\newtheorem{theorem}[definition]{Theorem}
\newtheorem{lemma}[definition]{Lemma}
\newtheorem{proposition}[definition]{Proposition}
\newtheorem{corollary}[definition]{Corollary}

\numberwithin{equation}{section}

\begin{document}

\title{Asymptotic structure of general metric spaces at infinity}

\author{Viktoriia Bilet and Oleksiy Dovgoshey}
\date{}
\maketitle

\begin{abstract}
Let $(X,d)$ be an unbounded metric space and $\tilde r=(r_n)_{n\in\mathbb N}$ be a scaling sequence of positive real numbers tending to infinity. We define the pretangent space $\Omega_{\infty, \tilde r}^{X}$ to $(X, d)$ at infinity as a metric space whose points are equivalence classes of sequences $(x_n)_{n\in\mathbb N}\subset X$ which tend to infinity with the speed of $\tilde r$. It is proved that the pretangent spaces $\Omega_{\infty, \tilde r}^{X}$ are complete for every unbounded metric space $(X, d)$ and every scaling sequence $\tilde r$. The finiteness conditions of $\Omega_{\infty, \tilde r}^{X}$ are found.
\end{abstract}


\bigskip


\section{Introduction}

Under the asymptotic structure of an unbounded metric space $(X,d)$ at infinity we mean the set of metric spaces which are the limits of rescaling metric spaces $\left(X, \frac{1}{r_n} d\right)$ for $r_n$ tending to infinity. The Gromov--Hausdorff convergence and the asymptotic cones are most often used for construction of such limits. Both of these approaches are based on higher-order abstractions (see, for example, \cite{Ro} for details), which makes them very powerful, but it does away the constructiveness. In this paper we propose a more elementary, sequential approach for describing the asymptotic structure of unbounded metric spaces at infinity.

Let $(X,d)$ be a metric space and let $\tilde r=(r_n)_{n\in\mathbb N}$ be a sequence of positive real numbers with $\mathop{\lim}\limits_{n\to\infty}r_{n}=\infty$. In what follows $\tilde{r}$ will be called a \emph{scaling sequence} and the formula $(x_n)_{n\in\mathbb N}\subset A$ will be mean that all elements of the sequence $(x_n)_{n\in\mathbb N}$ belong to the set~$A$.

\begin{definition}\label{D1.1}
Two sequences $\tilde{x}=(x_n)_{n\in \mathbb N}\subset X$ and $\tilde{y}=(y_n)_{n\in \mathbb N}\subset X$ are \emph{mutually stable} with respect to the scaling sequence $\tilde{r}=(r_n)_{n\in \mathbb N}$ if there is a finite limit
\begin{equation}\label{e1.1}
\lim_{n\to\infty}\frac{d(x_n,y_n)}{r_n} := \tilde{d}_{\tilde{r}}(\tilde{x},\tilde{y}) = \tilde{d}(\tilde{x}, \tilde{y}).
\end{equation}
\end{definition}

Let $p\in X.$ Denote by $Seq(X, \tilde r)$ the set of all sequences $\tilde x=(x_n)_{n\in\mathbb N}\subset X$ for which there is a finite limit

\begin{equation}\label{e1.2}
\lim_{n\to\infty}\frac{d(x_n, p)}{r_n} := \tilde{\tilde d}_{\tilde r}(\tilde x)
\end{equation}
and such that $\mathop{\lim}\limits_{n\to\infty}d(x_n, p)=\infty.$

\begin{definition}\label{D1.2}
A set $F\subseteq Seq(X, \tilde r)$ is \emph{self-stable} if any two $\tilde{x}, \tilde{y} \in F$ are mutually stable. $F$ is \emph{maximal self-stable} if it is self-stable and, for arbitrary $\tilde{y}\in Seq(X, \tilde r)$, we have either $\tilde{y}\in F$ or there is $\tilde{x}\in F$ such that $\tilde{x}$ and $\tilde{y}$ are not mutually stable.
\end{definition}

The maximal self-stable subsets of $Seq(X, \tilde r)$ will be denoted as $\tilde X_{\infty, \tilde r}$.

\begin{remark}\label{r1.3}
If $\tilde x=(x_n)_{n\in\mathbb N} \in Seq(X, \tilde r)$ and $p, b\in X,$
then the triangle inequality implies
\begin{equation}\label{e1.3}
\lim_{n\to\infty}\frac{d(x_n, p)}{r_n} = \lim_{n\to\infty}\frac{d(x_n, b)}{r_n}.
\end{equation}
In particular, $Seq(X, \tilde r),$  the self-stable subsets and the maximal self-stable subsets of $Seq(X, \tilde r)$ are invariant w.r.t. the choosing a point $p\in X$ in \eqref{e1.2}.
\end{remark}

Consider a function $\tilde{d}:\tilde{X}_{\infty,\tilde{r}} \times \tilde{X}_{\infty,\tilde{r}} \rightarrow \mathbb R$ satisfying \eqref{e1.1} for all $\tilde{x}$, $\tilde{y} \in \tilde{X}_{\infty,\tilde{r}}$. Obviously, $\tilde{d}$ is symmetric and nonnegative. Moreover, the triangle inequality for $d$ gives us the triangle inequality for $\tilde d$,
$$
\tilde{d}(\tilde{x},\tilde{y})\leq\tilde{d}(\tilde{x},\tilde{z})+\tilde{d}(\tilde{z},\tilde{y}).
$$
Hence $(\tilde{X}_{\infty, \tilde{r}},\tilde{d})$ is a pseudometric space.

Now we are ready to define the main object of our research.

\begin{definition}\label{D1.3}
Let $(X,d)$ be an unbounded metric space, let $\tilde{r}$ be a scaling sequence and let $\tilde{X}_{\infty, \tilde{r}}$ be a maximal self-stable subset of $Seq(X, \tilde r)$. The \emph{ pretangent space} to $(X, d)$ (at infinity, with respect to $\tilde{r}$) is the metric identification of the pseudometric space $(\tilde{X}_{\infty,\tilde{r}},\tilde{d})$.
\end{definition}


Since the notion of pretangent space is basic for the paper, we recall the metric identification construction. Define a relation $\equiv$ on $Seq(X, \tilde r)$ as
\begin{equation}\label{e1.4}
\left(\tilde x\equiv \tilde y\right)\Leftrightarrow \left(\tilde d_{\tilde{r}}(\tilde x, \tilde y)=0\right).
\end{equation}
The reflexivity an the symmetry of $\equiv$ are evident. Let $\tilde x, \tilde y, \tilde z\in Seq (X, \tilde r)$ and $\tilde x\equiv\tilde y$ and $\tilde y\equiv\tilde z.$ Then the inequality $$\limsup_{n\to\infty}\frac{d(x_n, z_n)}{r_n}\le\lim_{n\to\infty}\frac{d(x_n, y_n)}{r_n}+\lim_{n\to\infty}\frac{d(y_n, z_n)}{r_n}$$ implies $\tilde x\equiv\tilde z.$
Thus $\equiv$ is an equivalence relation.

Write $\Omega_{\infty,\tilde r}^{X}$ for the set of equivalence classes generated by the restriction of $\equiv$ on the set $\tilde{X}_{\infty, \tilde{r}}$. Using general properties of pseudometric spaces we can prove (see, for example, \cite{Kelley}) that the function $\rho \colon \Omega_{\infty, \tilde r}^{X} \times \Omega_{\infty, \tilde r}^{X} \to \mathbb{R}$ with
\begin{equation}\label{e1.5}
\rho(\alpha,\beta):=\tilde d_{\tilde r}(\tilde x, \tilde y), \quad \tilde x\in \alpha \in \Omega_{\infty, \tilde r}^{X}, \quad \tilde y\in \beta \in \Omega_{\infty, \tilde r}^{X},
\end{equation}
is a well-defined metric on~$\Omega_{\infty,\tilde r}^{X}$. The metric identification of $(\tilde X_{\infty, \tilde r}, \tilde d)$ is the metric space $(\Omega_{\infty, \tilde r}^{X}, \rho).$

Let $(n_k)_{k\in\mathbb N}\subset\mathbb N$ be a strictly increasing sequence. Denote by $\tilde r'$ the subsequence $(r_{n_k})_{k\in \mathbb N}$ of the scaling sequence $\tilde r=(r_n)_{n\in\mathbb N}$ and, for every $\tilde x=(x_n)_{n\in\mathbb N}\in Seq(X, \tilde r)$, write $\tilde x' := (x_{n_k})_{k\in\mathbb N}$. It is clear that $$\{\tilde x'\in Seq(X, \tilde r)\}\subseteq Seq(X, \tilde r')$$ and $\tilde{\tilde d}_{\tilde r'}(\tilde x') = \tilde{\tilde d}_{\tilde r}(\tilde x)$ holds for every $\tilde x\in Seq(X, \tilde r).$  Moreover, if sequences $\tilde x$, $\tilde y\in Seq(X, \tilde r)$ are mutually stable w.r.t. $\tilde r$, then $\tilde x'$ and $\tilde y'$ are mutually stable w.r.t. $\tilde r'$ and
\begin{equation}\label{e1.6}
\tilde d_{\tilde r}(\tilde x, \tilde y)=\tilde d_{\tilde r'}(\tilde x', \tilde y').
\end{equation}
By Zorn's lemma, for every $\tilde{X}_{\infty, \tilde{r}}\subseteq Seq(X, \tilde r)$, there is $\tilde X_{\infty,\tilde r'}\subseteq Seq(X, \tilde r')$ such that
\begin{equation}\label{e1.7}
\{\tilde x':\tilde x \in \tilde X_{\infty,\tilde r}\}\subseteq \tilde X_{\infty,\tilde r'}.
\end{equation}
Denote by $\varphi_{\tilde r'}$ the mapping from $\tilde X_{\infty,\tilde r}$ to $\tilde X_{\infty,\tilde r'}$ with $\varphi_{\tilde r'}(\tilde x)=\tilde x'$ for $\tilde x\in\tilde X_{\infty,\tilde r}.$ It follows from~\eqref{e1.6} that, after metric identifications, the mapping $\varphi_{\tilde r'}\colon \tilde X_{\infty, \tilde r} \to \tilde X_{\infty, \tilde r'}$ passes to an isometric embedding $em'\colon \Omega_{\infty,\tilde r}^{X} \rightarrow \Omega_{\infty,\tilde r'}^{X}$ such that the
diagram
\begin{equation}\label{e1.8}
\begin{CD}
\tilde X_{\infty, \tilde r} @>\varphi_{\tilde r'}>> \tilde X_{\infty, \tilde r'}\\
@V{\pi}VV @VV{\pi'}V \\
\Omega_{\infty, \tilde r}^{X} @>em'>> \Omega_{\infty, \tilde r'}^{X}
\end{CD}
\end{equation}
is commutative. Here $\pi$ and $\pi'$ are the corresponding natural projections,
\begin{equation}\label{e1.9}
\begin{aligned}
\pi(\tilde x) &:= \{\tilde y \in \tilde X_{\infty,\tilde r}\colon \tilde d_{\tilde r}(\tilde x, \tilde y)=0\}, \quad
\pi'(\tilde t) &:= \{\tilde y \in \tilde X_{\infty,\tilde r'}\colon \tilde d_{\tilde r'}(\tilde t, \tilde y)=0\}
\end{aligned}
\end{equation}
for all $\tilde x\in \tilde X_{\infty, \tilde r}$ and $\tilde t\in\tilde X_{\infty, \tilde r'}.$

\begin{definition}\label{D1.4}
Let $(X,d)$ be an unbounded metric space and let $\tilde{r}$ be a scaling sequence. A pretangent $\Omega_{\infty,\tilde r}^{X}$ is \emph{tangent} if $em'\colon \Omega_{\infty,\tilde r}^{X} \rightarrow \Omega_{\infty,\tilde r'}^{X}$ is surjective for every $\tilde X_{\infty, \tilde r'}.$
\end{definition}

It is can be proved that the following statements are equivalent.

$\bullet$ The metric space $\Omega_{\infty, \tilde r}^{X}$ is tangent.

$\bullet$ The mapping $em^{'}: \Omega_{\infty, \tilde r}^{X}\to\Omega_{\infty, \tilde r'}^{X}$ is an isometry for every $\tilde r'.$

$\bullet$ The set $\{\tilde x': \tilde x\in\tilde X_{\infty, \tilde r}\}$ is a maximal self-stable subset of $Seq(X, \tilde r')$ for every $\tilde r'.$

$\bullet$ The mapping $\varphi_{\tilde r'}:\tilde X_{\infty, \tilde r}\to\tilde X_{\infty, \tilde r'}$ is onto for every $\tilde r'.$

\medskip

In conclusion of this brief introduction we note that there exist other techniques which allow to investigate the asymptotic properties of metric spaces at infinity. As examples, we mention only the Gromov product which can be used to define a metric structure on the boundaries of hyperbolic spaces \cite{BS}, \cite{Sc}, the balleans theory~\cite{PZ} and the Wijsman convergence \cite{LechLev}, \cite{Wijs64}, \cite{Wijs66}.

\section{Basic properties of pretangent spaces}

Let us denote by $\tilde X_{\infty}$ the set of all sequences $(x_n)_{n\in\mathbb N}\subset X$ satisfying the limit relation $\mathop{\lim}\limits_{n\to\infty}d(x_n, p)=\infty$ with $p\in X.$ It is clear that $Seq(X, \tilde r)\subseteq\tilde X_{\infty}$ holds for every scaling sequence $\tilde r.$

\begin{proposition}\label{pr2.1}
Let $(X, d)$ be an unbounded metric space. Then the following statements hold.
\begin{enumerate}
\item[\rm(i)] The set $Seq(X, \tilde r)$ is nonempty for every scaling sequence $\tilde r.$
\item[\rm(ii)] For every $\tilde x \in\tilde X_{\infty},$ there exists a scaling sequence $\tilde r$ such that $\tilde x\in Seq(X, \tilde r)$.
\end{enumerate}
\end{proposition}
\begin{proof}
(i) Let $\tilde r=(r_n)_{n\in\mathbb N}$ be a scaling sequence and let $p\in X$. Let us denote by $\overline{B}\left(p, r_{n}^{\frac{1}{2}}\right)$ the closed ball
\begin{equation}\label{e2.1}
\{x\in X: d(x, p)\le r_{n}^{\frac{1}{2}}\}.
\end{equation}
Write
\begin{equation}\label{e2.2}
k_{n} := \sup\{d(x, p): x\in \overline{B}(p, r_n^{\frac{1}{2}})\},
\end{equation}
$n= 1,2, \ldots$. We can find $\tilde x=(x_n)_{n\in\mathbb N}\subset X$ such that
\begin{equation}\label{e2.3}
\lim_{n\to\infty}\frac{k_n}{d(x_n, p)}=1.
\end{equation}
Since $X$ is unbounded, the limit relation $\mathop{\lim}\limits_{n\to\infty}k_{n}=\infty$ holds. Consequently $\mathop{\lim}\limits_{n\to\infty}d(x_n, p)=\infty,$ i.e., $\tilde x\in\tilde X_{\infty}.$ It follows from \eqref{e2.1} and \eqref{e2.2} that the inequality $k_n\le r_{n}^{\frac{1}{2}}$ holds for every $n\in\mathbb N.$ The last inequality and \eqref{e2.3} imply $\mathop{\lim}\limits_{n\to\infty}\frac{d(x_n, p)}{r_n}=0.$ Thus, $\tilde x\in Seq(X, \tilde r).$

(ii) Let $\tilde x = (x_n)_{n\in\mathbb N}\in\tilde X_{\infty}$ and let $p\in X$. Then $\lim_{n\to\infty} d(x_n, p) = \infty$ holds. Define a sequence $\tilde r=(r_n)_{n\in\mathbb N}$ as
\begin{equation*}\label{Func}
r_{n}:=\begin{cases}
         d(x_n, p), & \mbox{if} $ $ x_n \ne p\\
         1,& \mbox{if}$ $ x_n = p\\
         \end{cases}
\end{equation*} for $n\in\mathbb N$. From $\tilde x\in\tilde X_{\infty}$ it follows that $\mathop{\lim}\limits_{n\to\infty}r_{n}=\infty.$ Hence $\tilde r$ is a scaling sequence. It is clear that $\tilde{\tilde d}_{\tilde r}(\tilde x)=1$. Thus, $Seq(X, \tilde r)\ni \tilde x$.
\end{proof}

For every unbounded metric space $(X, d)$ and every scaling sequence $\tilde{r}$ define the subset $\tilde X_{\infty, \tilde r}^{0}$ of the set $Seq(X, \tilde r)$ by the rule:

\begin{equation}\label{e2.4}
\left((z_n)_{n\in\mathbb N}\in\tilde X_{\infty, \tilde r}^{0}\right) \Leftrightarrow \left((z_n)_{n\in\mathbb N} \in \tilde X_{\infty} \quad \mbox{and} \quad \lim_{n\to\infty} \frac{d(z_n, p)}{r_n} = 0\right),
\end{equation}
where $p$ is a point of $X$.

Below we collect together some basic properties of the set $\tilde X_{\infty, \tilde r}^{0}.$

\begin{proposition}\label{pr2.2}
Let $(X, d)$ be an unbounded metric space and let $\tilde r$ be a scaling sequence. Then the following statements hold.
\begin{enumerate}
\item[\rm(i)] The set $\tilde X_{\infty, \tilde r}^{0}$ is nonempty.
\item[\rm(ii)] If $\tilde z\in \tilde X_{\infty, \tilde r}^{0},$ $\tilde y\in\tilde X_{\infty}$ and $\tilde d_{\tilde r}(\tilde z, \tilde y)=0,$ then $\tilde y\in\tilde X_{\infty, \tilde r}^{0}$ holds.
\item[\rm(iii)] If $F\subseteq Seq(X, \tilde r)$ is self-stable, then $\tilde X_{\infty, \tilde r}^{0}\cup F$ is also a self-stable subset of $Seq(X, \tilde r).$
\item[\rm(iv)] The set $\tilde X_{\infty, \tilde r}^{0}$ is self-stable.
\item[\rm(v)] The inclusion $\tilde X_{\infty, \tilde r}^{0}\subseteq\tilde X_{\infty, \tilde r}$ holds for every maximal self-stable subset $\tilde X_{\infty, \tilde r}$ of $Seq(X, \tilde r).$
\item[\rm(vi)] Let $\tilde z\in\tilde X_{\infty,\tilde r}^{0}$ and $\tilde x\in\tilde X_{\infty}.$ Then $\tilde x\in Seq (X, \tilde r)$ if and only if $\tilde x$ and $\tilde z$ are mutually stable. For $\tilde x\in Seq (X, \tilde r)$ we have $$\tilde{\tilde d}_{\tilde r}(\tilde x)=\tilde d_{\tilde r}(\tilde x, \tilde z).$$
\item[\rm(vii)] Denote by $\mathbf{\Omega_{\infty, \tilde r}^{X}}$ the set of all pretangent to $X$ at infinity (with respect to $\tilde r$) spaces. Then the membership relation
\begin{equation*}
\tilde X_{\infty, \tilde r}^{0}\in \bigcap_{\Omega_{\infty, \tilde r}^{X}\in\mathbf{\Omega_{\infty, \tilde r}^{X}}}\Omega_{\infty, \tilde r}^{X}
\end{equation*}
holds.
\end{enumerate}
\end{proposition}

\begin{proof}
(i) It follows from the proof of statement (i) in Proposition~\ref{pr2.1}.

\medskip

(ii) To prove $\tilde y\in\tilde X_{\infty, \tilde r}^{0}$ note that
\begin{equation*}
0\le\limsup_{n\to\infty}\frac{d(y_n, p)}{r_n}\le \tilde d_{\tilde r}(\tilde z, \tilde y)+\tilde{\tilde d}_{\tilde r}(\tilde z)=0.
\end{equation*}

\medskip

(iii) Let $p \in X$ and let $F\subseteq Seq(X, \tilde r)$ be self-stable. It is clear that $\tilde{\tilde d}_{\tilde r}(\tilde y)$ exists for every $\tilde y \in F \cup \tilde X_{\infty, \tilde r}^{0}$. Hence $F \cup \tilde X_{\infty, \tilde r}^{0}$ is self-stable if and only if $\tilde{z}$ and $\tilde{x}$ are mutually stable for all $\tilde x, \tilde z\in F\cup\tilde X_{\infty, \tilde r}^{0}.$. If $\tilde z, \tilde x\in F,$ then $\tilde z$ and $\tilde x$ are mutually stable by the condition. Suppose $\tilde x\in F$ and $\tilde z\in\tilde X_{\infty, \tilde r}^{0}.$ The inequalities
\begin{equation}\label{e2.5}
d(x_n, p)-d(z_n, p) \le d(x_n, z_n) \le d(x_n, p)+d(z_n, p)
\end{equation}
and the equality
$$
\lim_{n\to\infty} \frac{d(z_n, p)}{r_n} = 0
$$
imply the existence of $\tilde{d}_{\tilde{r}}(\tilde{x}, \tilde{z})$. The case $\tilde x, \tilde z\in\tilde X_{\infty, \tilde r}^{0}$ is similar. Thus the set $F\cup\tilde X_{\infty, \tilde r}^{0}$ is self-stable.

\medskip

(iv) This follows from (iii) with $F=\varnothing$.

\medskip

(v) Using statement (iii) with $F=\tilde X_{\infty, \tilde r}$ we see that $\tilde X_{\infty, \tilde r}^{0}\cup\tilde X_{\infty, \tilde r}$ is self-stable. Since $\tilde X_{\infty, \tilde r}$ is maximal self-stable, the equality $\tilde X_{\infty, \tilde r}^{0}\cup\tilde X_{\infty, \tilde r}=\tilde X_{\infty, \tilde r}$ holds. Thus, $\tilde X_{\infty, \tilde r}^{0}\subseteq\tilde X_{\infty, \tilde r}.$

\medskip

(vi) Suppose $\tilde x$ and $\tilde z$ are mutually stable. Then using \eqref{e2.5} we obtain
\begin{equation*}
\limsup_{n\to\infty}\frac{d(x_n, p)}{r_n}\le \tilde d_{\tilde r}(\tilde x, \tilde z)\le\liminf_{n\to\infty}\frac{d(x_n, p)}{r_n}.
\end{equation*}
Hence $\tilde x\in Seq(X, \tilde r).$ Similarly, if $\tilde x\in Seq(X, \tilde r),$ then we have
\begin{equation*}
\limsup_{n\to\infty}\frac{d(x_n, z_n)}{r_n}\le\tilde{\tilde d}_{\tilde r}(\tilde x)\le\liminf_{n\to\infty}\frac{d(x_n, z_n)}{r_n}.
\end{equation*}
Consequently $\tilde x$ and $\tilde z$ are mutually stable and $\tilde{\tilde d}_{\tilde r}(\tilde x)=\tilde d_{\tilde r}(\tilde x, \tilde z)$ holds.

\medskip

(vii) It follows from (ii), (v) and the definition of pretangent spaces.
\end{proof}

\begin{remark}
The set $\tilde X_{\infty, \tilde r}^{0}$ is invariant under replacing of $p\in X$ by an arbitrary point $b\in X$ in \eqref{e2.4}.
\end{remark}

\begin{lemma}\label{lem1.7}
Let $(X, d)$ be an unbounded metric space, $p\in X$ and $\tilde y\in\tilde X_{\infty}$, let $\tilde r$ be a scaling sequence and let $\tilde X_{\infty, \tilde r}$ be a maximal self-stable set. If $\tilde y$ and $\tilde x$ are mutually stable for every $\tilde x\in\tilde X_{\infty, \tilde r}$, then $\tilde y\in\tilde X_{\infty, \tilde r}$.
\end{lemma}
\begin{proof}
Suppose $\tilde y$ and $\tilde x$ are mutually stable for every $\tilde x\in\tilde X_{\infty, \tilde r}$. To prove $\tilde y\in\tilde X_{\infty, \tilde r}$ it suffices to show that there is a finite limit $\mathop{\lim}\limits_{n\to\infty}\frac{d(y_n, p)}{r_n}$ that follows from statement (vi) of Proposition~\ref{pr2.2}.
\end{proof}

\begin{lemma}\label{l2.7}
Let $(X, d)$ be an unbounded metric space and let $\tilde r$ be a scaling sequence. If $\tilde x$, $\tilde y$, $\tilde t \in \tilde X_{\infty}$ such that $\tilde x$ and $\tilde y$ are  mutually stable with respect to $\tilde r$ and $\tilde d_{\tilde r}(\tilde x, \tilde t)=0$, then $\tilde y$ and $\tilde t$ are  mutually stable with respect to $\tilde r$.
\end{lemma}

\begin{proof}
The statement follows from the equality $\tilde d_{r}(\tilde x, \tilde t)=0$ and the inequalities
\begin{multline*}
\tilde d_{\tilde r}(\tilde x, \tilde y) - \tilde d_{\tilde r}(\tilde x, \tilde t) \leq \liminf_{n\to \infty} \frac{d(y_n,t_n)}{r_n} \\
\leq \limsup_{n\to \infty} \frac{d(y_n,t_n)}{r_n} \leq \tilde d_{\tilde r}(\tilde x, \tilde y) + \tilde d_{\tilde r}(\tilde x, \tilde t).
\end{multline*}
\end{proof}

The set $\tilde X_{\infty, \tilde r}^{0}$ is a common distinguished point of all pretangent spaces $\Omega_{\infty,\tilde{r}}^X$ (with given scaling sequence $\tilde{r}$). We will consider the pretangent spaces to $(X,d)$ at infinity as the triples $(\Omega_{\infty,\tilde{r}}^X, \rho, \nu_{0})$, where $\rho$ is defined by~\eqref{e1.5} and $\nu_{0}: = \tilde X_{\infty, \tilde r}^{0}$. The point $\nu_{0}$ can be informally described as follows. The points of pretangent space $\Omega_{\infty, \tilde r}^{X}$ are infinitely removed from the initial space $(X, d),$ but $\Omega_{\infty, \tilde r}^{X}$ contains a unique point $\nu_{0}$ which is close to $(X, d)$ as much as possible.

Let $(X,d)$ be an unbounded metric space and let $\tilde{r}$ be a scaling sequence. Write
$\bar{\Omega}_{\tilde{r}, \infty}^X$ for the set of the equivalence classes generated by the relation $\equiv$ on the set $Seq(X, \tilde r)$ (see \eqref{e1.4}).
Let us consider the simple graph $G_{X, \tilde r}$ consisting of the vertex set $V(G_{X, \tilde r}) := \bar{\Omega}_{\tilde{r}, \infty}^X$ and the edge set $E(G_{X, \tilde r})$ defined by the rule:
$$
u \text{ and } v \text{ are adjecent}\quad \text{if and only if} \quad u\ne v\quad\text{and}
$$
$$
\quad \text{the limit} \quad \lim_{n\to\infty} \frac{d(x_n, y_n)}{r_n} \quad \text{exists for} \quad \tilde x\in u \quad \text{and} \quad \tilde y\in v.
$$
Recall that a \emph{clique} in a graph $G = (V,E)$ is a set $C \subseteq V$ such that every two distinct vertices of $C$ are adjacent. A \emph{maximal clique} is a clique $C_1$ such that the inclusion
$$
V(C_1) \subseteq V(C)
$$
implies the equality $V(C_1) = V(C)$ for every clique $C$ in $G$.

\begin{theorem}\label{t2.12}
Let $(X,d)$ be an unbounded metric space and let $\tilde{r}$ be a scaling sequence. A set $C \subseteq \bar{\Omega}_{\tilde{r}, \infty}^X$ is a maximal clique in $G_{X, \tilde r}$ if and only if there is a pretangent spaces $\Omega_{\infty, \tilde r}^{X}$ such that $C=\Omega_{\infty, \tilde r}^{X}.$
\end{theorem}
\begin{proof}
Lemma~\ref{lem1.7} and Lemma~\ref{l2.7} imply the equality
\begin{equation}\label{e2.9}
\{\tilde x\in\tilde X_{\infty, \tilde r}:\tilde d_{\tilde r}(\tilde x, \tilde y)=0\}=\{\tilde x\in Seq(X, \tilde r): \tilde d_{\tilde r}(\tilde x, \tilde y)=0\}
\end{equation}
for every $\tilde y\in\tilde X_{\infty, \tilde r}$ and every $\tilde X_{\infty, \tilde r}.$ Since, for every $\tilde y\in Seq(X, \tilde r),$ there is $\tilde X_{\infty, \tilde r}$ such that $\tilde X_{\infty, \tilde r}\ni\tilde y,$ equality \eqref{e2.9} implies
\begin{equation}\label{e2.10}
\bar{\Omega}_{\infty, \tilde r}^{X}=\bigcup_{\Omega_{\infty, \tilde r}^{X} \in \mathbf{\Omega_{\infty, \tilde r}^{X}}}\Omega_{\infty, \tilde r}^{X},
\end{equation}
where $\mathbf{\Omega_{\infty, \tilde r}^{X}}$ is the set of all spaces which are pretangent to $X$ at infinity with respect to $\tilde r.$ Now the theorem follows from the definitions of the pretangent spaces and the maximal cliques.
\end{proof}

Theorem~\ref{t2.12} gives some grounds for calling the graph $G_{X, \tilde r}$ a \emph{net of pretangent spaces}.

In the next proposition we follow terminology used in \cite{BM}.
Recall only that a vertex $v$ of a graph $G=(V, E)$ is a \emph{dominating vertex} if $\{u, v\}\in E$ for all $u\in V\setminus\{v\}.$

\begin{proposition}\label{Pr2.8}
Let $(X, d)$ be an unbounded metric space and let $\tilde r$ be a scaling sequence. Then the following statements hold.
\begin{enumerate}
\item[\rm(i)] The vertex $\nu_{0}=\tilde X_{\infty, \tilde r}^{0}$ is a dominating vertex of the graph $G_{X, \tilde r}$.
\item[\rm(ii)] $G_{X, \tilde r}$ is complete if and only if there is a unique pretangent space $\Omega_{\infty, \tilde r}^{X}$.
\item[\rm(iii)] $G_{X, \tilde r}$ is a star if and only if
$$
\sup\{\left|\Omega_{\infty, \tilde r}^{X}\right|: \Omega_{\infty, \tilde r}^{X}\in\mathbf{\Omega_{\infty, \tilde r}^{X}}\}=2,
$$
where $\left|\Omega_{\infty, \tilde r}^{X}\right|$ is the cardinal number of $\Omega_{\infty, \tilde r}^{X}$.
\end{enumerate}
\end{proposition}

The proof is simple. Note only that (i) follows from statement (vii) of Proposition~\ref{pr2.2}.

\section{Completeness of pretangent spaces}

It is well know that the Gromov--Hausdorff limits and the asymptotic cones of metric spaces are always complete. The quasi-metrics on the boundaries of hyperbolic spaces are also complete (see, for example, Proposition~6.1 in~\cite{Sc}). The goal of this section is to show that every pretangent space is complete. For the proof of this fact we shall use the following lemmas.

\begin{lemma}\label{l3.1}
Let $(X, d)$ be an unbounded metric space, $\tilde r$ be a scaling sequence, $\tilde X_{\infty, \tilde r}$ be maximal self-stable, $\tilde x \in\tilde X_{\infty}$ and let $(\tilde\gamma^{m})_{m\in\mathbb N} \subset\tilde X_{\infty, \tilde r}$ such that $\tilde\gamma^{m}$ and $\tilde x$ are mutually stable for every $m\in\mathbb N$ and let
\begin{equation}\label{e3.1}
\lim_{m\to\infty}\tilde d(\tilde x, \tilde\gamma^{m})=0.
\end{equation}
Then $\tilde x$ belongs to $\tilde X_{\infty, \tilde r}$.
\end{lemma}

\begin{proof}
By Lemma~\ref{lem1.7} $\tilde x\in\tilde X_{\infty, \tilde r}$ if and only if for every $\tilde y\in\tilde X_{\infty, \tilde r}$ there is a finite limit
\begin{equation}\label{e3.2}
\lim_{n\to\infty}\frac{d(x_n, y_n)}{r_n}.
\end{equation}
Let $\tilde y\in\tilde X_{\infty, \tilde r}$. It follows from the triangle inequality for $\tilde d$ that
\begin{equation}\label{e3.3}
|\tilde d(\tilde y, \tilde\gamma^{m_1})-\tilde d(\tilde y, \tilde\gamma^{m_2})|\le\tilde d(\tilde\gamma^{m_1}, \tilde\gamma^{m_2})\le\tilde d(\tilde x, \tilde\gamma^{m_1})+\tilde d(\tilde x, \tilde\gamma^{m_2})
\end{equation}
for all $m_1, m_2\in\mathbb N.$ Now \eqref{e3.1} and \eqref{e3.3} imply that $(\tilde d(\tilde y, \tilde\gamma^{m}))_{m\in\mathbb N}$ is a Cauchy sequence in $\mathbb R.$ Consequently, there is a finite limit $\mathop{\lim}\limits_{m\to\infty}\tilde d(\tilde y, \tilde\gamma^{m}).$ We claim that limit \eqref{e3.2} exists and
\begin{equation}\label{e3.4}
\lim_{n\to\infty}\frac{d(x_n, y_n)}{r_n}=\lim_{m\to\infty}\tilde d(\tilde y, \tilde\gamma^{m}).
\end{equation}
This statement holds if and only if
\begin{equation}\label{e3.5}
\limsup_{n\to\infty}\frac{d(x_n, y_n)}{r_n}=\lim_{m\to\infty}\tilde d(\tilde y, \tilde\gamma^{m})
\end{equation}
and
\begin{equation}\label{e3.6}
\liminf_{n\to\infty}\frac{d(x_n, y_n)}{r_n}=\lim_{m\to\infty}\tilde d(\tilde y, \tilde\gamma^{m}).
\end{equation}
Equality \eqref{e3.5} holds if and only if
\begin{equation*}
\lim_{m\to\infty}\left|\tilde d(\tilde y, \tilde\gamma^{m})-\limsup_{n\to\infty}\frac{d(x_n, y_n)}{r_n}\right|=0.
\end{equation*}
It is clear that
\begin{equation*}
\lim_{m\to\infty}\left|\tilde d(\tilde y, \tilde\gamma^{m})-\limsup_{n\to\infty}\frac{d(x_n, y_n)}{r_n})\right|
=\lim_{m\to\infty}\limsup_{n\to\infty}\left|\frac{d(y_n, \gamma_{n}^{m})}{r_n}-\frac{d(x_n, y_n)}{r_n}\right|
\end{equation*}
\begin{equation*}
\le\lim_{m\to\infty}\limsup_{n\to\infty}\frac{d(\gamma_{n}^{m}, x_n)}{r_n}=\lim_{m\to\infty}\tilde d(\tilde\gamma^{m}, \tilde x)=0,
\end{equation*}
where $(\gamma_{n}^{m})_{n\in\mathbb N}=\tilde\gamma^{m}$. Equality \eqref{e3.5} follows. Equality \eqref{e3.6} can be proved similarly.
\end{proof}

\begin{lemma}\label{l3.2}
Let $(X, d)$ be an unbounded metric space and let $\tilde r=(r_n)_{n\in\mathbb N}$ be a scaling sequence. Then for every $\tilde x=(x_n)_{n\in\mathbb N}\subset X$ there is $\tilde y=(y_n)_{n\in\mathbb N}\in\tilde X_{\infty}$ such that $\mathop{\lim}\limits_{n\to\infty}\frac{d(x_n, y_n)}{r_n}=0.$
\end{lemma}

\begin{proof}
Let $\tilde z=(z_n)_{n\in\mathbb N}\in\tilde X_{\infty, \tilde r}^{0}$ and let $p\in X$. For every $\tilde{x} \subset X$ define a sequence $\tilde{y} = (y_n)_{n\in\mathbb N} \subset X$ by the rule
\begin{equation}\label{e3.7}
y_{n} := \begin{cases}
x_n, & \mbox{if } d(x_n, p) \ge d(z_n, p)\\
z_n, & \mbox{if } d(x_n, p) < d(z_n, p). \\
\end{cases}
\end{equation}
It follows from \eqref{e3.7} that the inequality
\begin{equation}\label{e3.8}
d(y_n, p)\ge d(z_n, p)
\end{equation}
holds for every $n\in\mathbb N.$ Since we have
$\tilde z\in\tilde X_{\infty, \tilde r}^{0}\subset\tilde X_{\infty},$
inequality \eqref{e3.8} implies $\tilde y=(y_n)_{n\in\mathbb N}\in\tilde X_{\infty}.$ Moreover,
from \eqref{e3.8} we also have
\begin{equation}\label{e3.9}
0\le d(x_n, y_n)\le 2d(z_n, p).
\end{equation}
The equality
$$
\lim_{n\to\infty}\frac{d(x_n, y_n)}{r_n}=0
$$
follows from \eqref{e3.9} and $\lim\limits_{n\to\infty}\frac{d(z_n, p)}{r_n}=0$.
\end{proof}

\begin{theorem}\label{th3.3}
Let $(X, d)$ be an unbounded metric space. Then all pretngent spaces to $(X, d)$ at infinity are complete.
\end{theorem}

\begin{proof}
Let $p\in X,$ let $\tilde r=(r_n)_{n\in\mathbb N}$ be a scaling sequence and let $\tilde X_{\infty, \tilde r}$ be a maximal self-stable set with metric identification $(\Omega_{\infty, \tilde r}^{X}, \rho).$ The metric space $(\Omega_{\infty, \tilde r}^{X}, \rho)$ is complete if and only if the pseudometric space $(\tilde X_{\infty, \tilde r}, \tilde d)$ is complete, i.e., for every Cauchy sequence $(\tilde\gamma^{m})_{m\in\mathbb N}\subset\tilde X_{\infty, \tilde r}$ there is $\tilde x=(x_n)_{n\in\mathbb N}\in\tilde X_{\infty, \tilde r}$ such that
\begin{equation}\label{e3.10}
\lim_{m\to\infty}\tilde d(\tilde x, \tilde\gamma^{m})=0.
\end{equation}
By Lemma~\ref{l3.1} if \eqref{e3.10} holds with some $\tilde x\in\tilde X_{\infty},$ then $\tilde x\in\tilde X_{\infty, \tilde r}.$ Let $(\tilde\gamma^{m})_{m\in\mathbb N}$ be a Cauchy sequence in $(\tilde X_{\infty, \tilde r}, \tilde d).$ We first find $\tilde x=(x_n)_{n\in\mathbb N}\subset X$ for which \eqref{e3.10} holds. Then, using Lemma~\ref{l3.2}, we obtain $\tilde x\in\tilde X_{\infty}$ satisfying \eqref{e3.10}. Let $(\varepsilon)_{k\in\mathbb N}\subset (0, \infty)$ be an decreasing sequence such that
\begin{equation}\label{e3.11}
\sum_{k=1}^{\infty}\varepsilon_{k}<\infty.
\end{equation}
There is $(m_k)_{k\in\mathbb N}\subset\mathbb N$ such that, for all $k\in\mathbb N,$ $m_{k+1}>m_{k}$ and $\tilde d(\tilde\gamma^{m}, \tilde\gamma^{m_k})\le\varepsilon_{k}$
holds whenever $m\ge m_{k}.$ Now we construct $\tilde x=(x_n)_{n\in\mathbb N}\subset X$ such that $\tilde x$ and $\tilde\gamma^{m_k}$ are mutually stable for every $k\in\mathbb N$ and
\begin{equation}\label{e3.12}
\lim_{k\to\infty}\tilde d(\tilde\gamma^{m_k}, \tilde x)=0.
\end{equation}
For every $m\in\mathbb N$ we set $\tilde\gamma^{m}=(\gamma_{n}^{m})_{n\in\mathbb N}.$ Let $(N_k)_{k\in\mathbb N}\subset\mathbb N$ and $\tilde\beta^{k}=(\beta_{n}^{k})_{n\in\mathbb N}\in\tilde X_{\infty}$ be inductively defined by the rule: if $k=1,$ then $N_1=1$ and $(\beta_{n}^{1})_{n\in\mathbb N}=(\gamma_{n}^{m_1})_{n\in\mathbb N};$ if $k\ge 2,$ then $N_k$ is the smallest $l\in\mathbb N$ which satisfies the inequalities $l>N_{k-1}$ and
\begin{equation}\label{e3.13}
\beta_{n}^{k}:=\begin{cases}
         \beta_{n}^{1} & \mbox{if} $ $ N_1\le n<N_2,\\
         \beta_{n}^{2}& \mbox{if}$ $  N_2\le n<N_1,\\
         \cdots &\cdots\cdots\cdots\cdots,\\
          \beta_{n}^{k-1}& \mbox{if}$ $  N_{k-1}\le n<N_k,\\
          \gamma_{n}^{m_k}& \mbox{if}$ $  n\ge N_k.\\
         \end{cases}
\end{equation}
Define $\tilde x=(x_n)_{n\in\mathbb N}$ as
\begin{equation}\label{e3.14}
x_n:=\beta_{n}^{k} \quad\mbox{for}\quad n\in[N_k, N_{k+1}), k=1,2,3,\ldots .
\end{equation}
It follows from \eqref{e3.12} and \eqref{e3.13} that
\begin{equation}\label{e3.15}
\lim_{k\to\infty}\frac{d(\beta_{n}^{k}, \gamma_{n}^{m_k})}{r_n}=0
\end{equation}
for every $k\in\mathbb N,$ and
\begin{equation}\label{e3.16}
\frac{d(\beta_{n}^{k}, \beta_{n}^{k-1})}{r_n}<2\varepsilon_{k}
\end{equation}
for all $n, k\in\mathbb N.$ Limit relation \eqref{e3.15} implies that \eqref{e3.12} holds
if and only if
\begin{equation}\label{e3.17}
\lim_{k\to\infty}\limsup_{n\to\infty}\left(\frac{d(x_n, \beta_{n}^{k})}{r_n}\right)=0.
\end{equation}
Using \eqref{e3.14}, we see that for every $n\in\mathbb N$ there is $K(n)\in\mathbb N$ such that
\begin{equation*}
\frac{d(x_n, \beta_{n}^{k})}{r_n}=\frac{d(\beta_{n}^{K(n)}, \beta_{n}^{k})}{r_n}.
\end{equation*}
If $k$ is given, then, for sufficiently large $n,$ the inequality $K(n)>k$ holds. Consequently
by \eqref{e3.16} we have
\begin{equation}\label{e3.18}
\frac{d(\beta_{n}^{K(n)}, \beta_{n}^{k})}{r_n}\le\sum_{i=0}^{K(n)-k}\frac{d(\beta_{n}^{k+i+1}, \beta_{n}^{k+i})}{r_n}\le 2\sum_{i=k}^{\infty}\varepsilon_{i}.
\end{equation}
Inequalities \eqref{e3.11} and \eqref{e3.18} imply \eqref{e3.17}.
\end{proof}

If $Y$ is an unbounded subspace of metric space $X$ and for a scaling sequence $\tilde r$ and maximal self-stable $\tilde X_{\infty, \tilde r}$ and $\tilde Y_{\infty, \tilde r}$ we have $\tilde Y_{\infty, \tilde r}\subseteq\tilde X_{\infty, \tilde r},$ then there is a unique isometric embedding $in_{Y}:\Omega_{\infty, \tilde r}^{Y}\to\Omega_{\infty, \tilde r}^{X}$ such that the diagram \begin{equation*}
\begin{CD}
\tilde Y_{\infty, \tilde r} @>In_{\tilde Y}>> \tilde X_{\infty, \tilde r}\\
@V{\pi_{Y}}VV @VV{\pi_{X}}V \\
\Omega_{\infty, \tilde r}^{Y} @>in_{Y}>> \Omega_{\infty, \tilde r}^{X}
\end{CD}
\end{equation*}
is commutative, where $\pi_{Y}$ and $\pi_{X}$ are the natural projections and $In_{\tilde Y}(\tilde y)=\tilde y$ for every $\tilde y\in\tilde Y_{\infty, \tilde r}.$

Let $X$ and $Y$ be a metric spaces. Recall that a map $f: X\to Y$ is called \emph{closed} if the image of each closed set is closed.

\begin{corollary}\label{c3.4}
Let $(X, d)$ be an unbounded metric space, $Y$ be an unbounded subspace of $X,$ $\tilde r$ be a scaling sequence and let $\Omega_{\infty, \tilde r}^{X}$ and $\Omega_{\infty, \tilde r}^{Y}$ be pretangent spaces such that for corresponding $\tilde X_{\infty, \tilde r}$ and $\tilde Y_{\infty, \tilde r}$ we have $\tilde X_{\infty, \tilde r}\supseteq\tilde Y_{\infty, \tilde r}.$ Then the isometric embedding $in_{Y}:\Omega_{\infty, \tilde r}^{Y}\to\Omega_{\infty, \tilde r}^{X}$ is closed.
\end{corollary}

\begin{proof}
The map $in_{Y}$ is an isometric embedding. Hence, $in_{Y}$ is closed if and only if the set $in_{Y}\left(\Omega_{\infty, \tilde r}^{Y}\right)$ is a closed subset of $\Omega_{\infty, \tilde r}^{X}.$ The space $\Omega_{\infty, \tilde r}^{X}$ is complete by Theorem~\ref{th3.3}. Since a metric space is complete if and only if this space is closed in every its superspace (see, for example, \cite[Theorem~10.2.2]{Sear}), $in_{Y}\left(\Omega_{\infty, \tilde r}^{Y}\right)$ is closed in $\Omega_{\infty, \tilde r}^{X}.$
\end{proof}

\section{When are the pretangent spaces finite?}

The main goal of this section is to find conditions under which the ine\-quality $\left|\Omega_{\infty, \tilde r}^{X}\right|\le n$ holds, with given $n\in\mathbb N,$ for all pretangent at infinity spaces $\Omega_{\infty, \tilde r}^{X}.$

\begin{lemma}\label{Pr2}
Let $(X, d)$ be an unbounded metric space. Then there exists a pretangent space $\Omega_{\infty, \tilde r}^{X}$ such that $\left|\Omega_{\infty, \tilde r}^{X}\right|\ge 2.$
\end{lemma}

The proof of the lemma is similar to the proof of statement (ii) from Proposition~\ref{pr2.1}.

The following lemma is an analog of Lemma 5 from \cite{DAK}.

\begin{lemma}\label{l4.2}
Let $(X,d)$ be an unbounded metric space, $p\in X,$ let $\mathbf{\mathfrak{B}}$ be a countable subset of $\tilde X_{\infty}$ and let $\tilde r=(r_n)_{n\in\mathbb N}$ be a scaling sequence.
Suppose that
\begin{equation}\label{e4.14}
\limsup_{n\to\infty}\frac{d(b_n, p)}{r_n}<\infty
\end{equation}
holds for every $\tilde
b=(b_n)_{n\in\mathbb N}\in\mathbf{\mathfrak{B}}$. Then there is a strictly increasing sequence $(n_k)_{k\in\mathbb N}\subset\mathbb N$
such that the family
$$
\mathbf{\mathfrak{B'}}:=\{\tilde b'=(b_{n_k})_{k\in\mathbb N}: \tilde b \in \mathbf{\mathfrak{B}}\}
$$
is self-stable at infinity with respect to $\tilde r'=(r_{n_k})_{k\in\mathbb N}.$
\end{lemma}
\begin{proof}
It is sufficient to consider the case when $\mathfrak B$ is countably infinite. Then the set of all ordered pairs $(\tilde b, \tilde x)\in\mathfrak B^{2}$ can be enumerated as $(\tilde b^{1}, \tilde x^{1}), (\tilde b^{2}, \tilde x^{2}), \ldots$. The triangle inequality and \eqref{e4.14} imply
\begin{equation*}
\sup_{n\in\mathbb N}\frac{d(b_{n}^{j}, x_{n}^{j})}{r_n}<\infty
\end{equation*}
for each pair $(\tilde b^{j}, \tilde x^{j})\in\mathfrak B^{2},$ $\tilde b^{j}=(b_{n}^{j})_{n\in\mathbb N}$ and $\tilde x^{j}=(x_{n}^{j})_{n\in\mathbb N}$. In particular, we have
\begin{equation*}
\sup_{n\in\mathbb N}\frac{d(b_{n}^{1}, x_{n}^{1})}{r_n}<\infty.
\end{equation*}
Since every bounded, infinite sequence contains a convergent subsequence, there is a strictly increasing sequence $\tilde n^{1}=(n_{k}^{1})_{k\in\mathbb N} \subseteq \mathbb N$ such that
\begin{equation*}
\lim_{k\to\infty}\frac{d(b_{n_{k}^{1}}^{1}, x_{n_{k}^{1}}^{1})}{r_{n_k^{1}}}, \quad
\lim_{k\to\infty}\frac{d(x_{n_{k}^{1}}^{1}, p)}{r_{n_k^{1}}} \quad\mbox{and}\quad
\lim_{k\to\infty}\frac{d(b_{n_{k}^{1}}^{1}, p)}{r_{n_k^{1}}}
\end{equation*}
are finite. Hence, the sequences $(b_{n_{k}^{1}}^{1})_{k\in\mathbb N}$ and $(x_{n_{k}^{1}}^{1})_{k\in\mathbb N}$ are mutually stable with respect to $(r_{n_k^{1}})_{k\in\mathbb N}$. Analogously, by induction, we can prove that for every integer $i\ge 2$ there is a subsequence $\tilde n_{i}=(n_{k}^{i})_{k\in\mathbb N}$ of sequence $\tilde n_{i-1}$ such that $(b_{n_{k}^{i}}^{i})_{k\in\mathbb N}$ and $(x_{n_{k}^{i}}^{i})_{k\in\mathbb N}$ are mutually stable with respect to $(r_{n_k^{i}})_{k \in \mathbb N}$. Using Cantor's diagonal construction, write $\tilde r' := (r_{n_k^{k}})_{k\in\mathbb N}$ and, for every $\tilde b=(b_n)_{n\in\mathbb N}\in\mathfrak B$, define $\tilde b'$ as $\tilde b' := (b_{n_k^{k}})_{k \in\mathbb N}$. Then the family $\mathfrak B':=\{\tilde b': \tilde b\in\mathfrak B\}$ is self-stable at infinity with respect to $\tilde r'$.
\end{proof}

Let $(X,d)$ be an unbounded metric space and let $p$ be a point of $X$. Denote by $X^{n}$ the set of all $n$-tuples $x=(x_1,\ldots, x_n)$ with $x_k\in X$ for $k=1, \ldots, n,$ $n\ge 2$ and define the function $F_n: X^{n}\to\mathbb R$ as
\begin{multline}\label{e4.13}
F_n(x_1, \ldots, x_n):=\\
\begin{cases}
0& \mbox{if } (x_1, \ldots, x_n) = (p, \ldots, p)\\
\dfrac{\min\limits_{1\le k\le n}d(x_k, p)\prod\limits_{1\le k<l\le n} d(x_k, x_l)} {\left(\max\limits_{1\le k\le n}d(x_k, p)\right)^{\frac{n(n-1)}{2}+1}} & \mbox{otherwise}.
\end{cases}
\end{multline}

\begin{theorem}\label{t4.3}
Let $(X, d)$ be an unbounded metric space and let $n\ge 2$ be an integer number. Then the inequality
\begin{equation}\label{e4.15}
\left|\Omega_{\infty, \tilde r}^{X}\right|\le n
\end{equation}
holds for every $\Omega_{\infty, \tilde r}^{X}$ if and only if
\begin{equation}\label{e4.16}
\lim_{x_1, \ldots, x_n\to\infty}F_n(x_1, \ldots,  x_n)=0.
\end{equation}
\end{theorem}

\begin{proof}
Let \eqref{e4.15} hold for all pretangent spaces $\Omega_{\infty, \tilde r}^{X}$. Suppose $\tilde x^{i}=(x_{m}^{i})_{m\in\mathbb N}\in\tilde X_{\infty}$, $i=1, \ldots, n$ such that
\begin{equation}\label{e4.17}
\lim_{m\to\infty}F_{n}(x_{m}^{1},\ldots, x_{m}^{n}) = \limsup_{x_1, \ldots, x_n \to \infty}F_{n}(x_1, \ldots, x_n) > 0.
\end{equation}
If $\tilde r=(r_m)_{m\in\mathbb N}$ is a scaling sequence with
$$
r_m=\max\{1, d(x_{m}^{1}, p), \ldots, d(x_{m}^{n}, p)\}
$$
for every $m\in\mathbb N$, where $p$ is a point of $X$ in definition \eqref{e4.13}, then the inequality
$$
\limsup_{m\to\infty}\frac{d(x_{m}^{k}, p)}{r_m}\le 1
$$
holds for every $k\in\{1, \ldots, n\}$. Using Lemma~\ref{l4.2} we may suppose that the family $\{\tilde x^{1}, \ldots, \tilde x^{n}\}$ is self-stable with respect to $\tilde r$. Now \eqref{e4.17} and \eqref{e4.13} imply
$$
\tilde{\tilde d}_{\tilde r}(\tilde x^{k})>0\quad\mbox{and}\quad \tilde d_{\tilde r}(\tilde x^{k}, \tilde x^{j})>0
$$
for all distinct $k, j\in\{1, \ldots, n\}$. Adding $\tilde z\in\tilde X_{\infty, \tilde r}^{0}$ to the family $\{\tilde x^{1}, \ldots, \tilde x^{n}\}$ we see that the family $\{\tilde z, \tilde x^{1}, \ldots, \tilde x^{n}\}$ is self-stable by statement (iii) of Proposition~\ref{pr2.2}. Consequently there is $\Omega_{\infty, \tilde r}^{X}$ with $\left|\Omega_{\infty, \tilde r}^{X}\right|\ge n+1$, contrary to \eqref{e4.15}. Equality \eqref{e4.16} follows.

To prove the converse statement it suffices to consider some different $n+1$ points $\nu_{0}, \nu_{1}, \ldots, \nu_{n}\in\Omega_{\infty, \tilde r}^{X}$ such that $\nu_{0}=\tilde X_{\infty, \tilde r}^{0},$ (see \eqref{e1.8}). Then, for the sequences $\tilde x^{1}, \ldots, \tilde x^{n}$ with
$$
\pi(\tilde x^{k})=\nu_{k},\,\, \tilde x^{k}=(x_{m}^{k})_{m\in\mathbb N},\,\, k\in\{1, \ldots, n\},
$$
we obtain
\begin{equation}\label{e4.18}
\lim_{m\to\infty}F_{n}(x_{m}^{1}, \ldots, x_{m}^{n}) = \frac{\min\limits_{1\le k\le n} \rho(\nu_0, \nu_k) \prod\limits_{1\le k<l\le n} \rho(\nu_k, \nu_l)} {\left(\mathop{\max}\limits_{1\le k\le n} \rho(\nu_k, \nu_0)\right)^{\frac{n(n-1)}{2}+1}}\ne 0.
\end{equation}
\end{proof}

\begin{corollary}\label{c4.4}
Let $(X, d)$ be an unbounded metric space and let $n\ge 2$ be an integer number. Suppose $\mathop{\lim}\limits_{x_1, \ldots, x_n\to\infty}F_{n}(x_1, \ldots, x_n)=0$ holds. Then every pretangent space $\Omega_{\infty, \tilde r}^{X}$ with $\left|\Omega_{\infty, \tilde r}^{X}\right|=n$ is tangent.
\end{corollary}

\begin{theorem}\label{t4.5}
Let $(X, d)$ be an unbounded metric space and let $n\ge 2$ be an integer number such that the inequality
\begin{equation}\label{e4.19}
\left|\Omega_{\infty, \tilde r}^{X}\right|\le n
\end{equation}
holds for every $\Omega_{\infty, \tilde r}^{X}.$ Then the following statements are equivalent.
\begin{enumerate}
\item[\rm(i)] For every scaling sequence $\tilde r,$ the function $\rho^{0}: V(G_{X, \tilde r})\to\mathbb R$ defined as $$\rho^{0}(\nu)=\tilde{\tilde d}_{\tilde r}(\tilde x), \,\, \tilde x\in\nu\in V(G_{X, \tilde r})$$ is injective.
\item[\rm(ii)] For every scaling sequence $\tilde r,$ the inequality
\begin{equation}\label{e4.20}
\left|V(G_{X, \tilde r})\right|\le n
\end{equation}
holds.
\end{enumerate}
If, in addition, $\tilde r$ is a scaling sequence such that
\begin{equation}\label{e4.21}
\rho^{0}(\nu_1)=\rho^{0}(\nu_2)
\end{equation}
for some distinct $\nu_1, \nu_2\in V(G_{X, \tilde r})$, then we have
\begin{equation}\label{e4.22}
\left|V(G_{X, \tilde r})\right|\ge \mathfrak{c},
\end{equation}
where $\mathfrak{c}$ is the continuum.
\end{theorem}

\begin{remark}\label{r4.6}
For every vertex $\nu\in V(G_{X, \tilde r})$ there is a pretangent space $\left(\Omega_{\infty, \tilde r}^{X}, \rho\right)$ such that $\nu\in\Omega_{\infty, \tilde r}^{X}.$ Then $\rho^{0}(\nu)$ is the distance from $\nu$ to the distinguished point $\nu_{0}=\tilde X_{\infty, \tilde r}^{0}$ in the metric space $\left(\Omega_{\infty, \tilde r}^{X}, \rho\right),$ $\rho^{0}(\nu)=\rho(\nu_{0}, \nu).$
\end{remark}

\begin{proof}[Proof of Theorem \ref{t4.5}.]
(i)$\Rightarrow$(ii) Let $\rho^{0}: V(G_{X, \tilde r})\to\mathbb R$ be injective for every $\tilde r.$ Suppose (ii) does not hold. Then there is a scaling sequence $\tilde r_{1}$ such that
$$\left|V(G_{X, \tilde r_{1}})\right|\ge n+1.$$ Consequently, we can find $\tilde x^{0}\in\tilde X_{\infty, \tilde r_1}^{0}$ and $\tilde x_{1}, \ldots, \tilde x_{n}\in Seq(\tilde X, \tilde r_{1})$ such that
\begin{equation}\label{e4.23}
0=\tilde{\tilde d}_{\tilde r_{1}}(\tilde x_{0})<\tilde{\tilde d}_{\tilde r_{1}}(\tilde x_{k})<\tilde{\tilde d}_{\tilde r_{1}}(\tilde x_{k+1})<\infty
\end{equation}
for every $k\in\{1, \ldots, n-1\}.$ By Lemma~\ref{l4.2} there is an infinite subsequence $\tilde r_{1}^{'}$ of the sequence $\tilde r_{1}$ such that the set $\{\tilde x'_{0}, \tilde x'_{1}, \ldots, \tilde x'_{n}\}$ is self-stable. Let $\tilde X_{\infty, \tilde r}\supseteq\{\tilde x'_{0}, \tilde x'_{1}, \ldots, \tilde x'_{n}\}$ and let $(\Omega_{\infty, \tilde r'_{1}}^{X}, \rho)$ be the metric identification of $\tilde X_{\infty, \tilde r'}$. Write $\nu_{i}=\pi(\tilde x'_{i}),$ $i=0, \ldots, n,$ where $\pi: \tilde X_{\infty, \tilde r'_{1}}\to\Omega_{\infty, \tilde r'_{1}}^{X}$ is the natural projection. Now \eqref{e4.23} implies that
$$
0=\rho(\nu_{0}, \nu_{0})<\rho(\nu_{0}, \nu_{1})<\ldots<\rho(\nu_{0}, \nu_{n}).
$$
Consequently $\nu_{0}, \ldots, \nu_{n}$ are distinct points of the space $\Omega_{\infty, \tilde r}^{X}$. Thus
$$
\left|\Omega_{\infty, \tilde r'_{1}}^{X}\right| \ge n+1,
$$
which contradicts \eqref{e4.19} with $\tilde r=\tilde r'_{1}.$

(ii)$\Rightarrow$(i) Suppose now that (i) does not hold. Then there exist $\tilde r$ and $\nu_1, \nu_2\in V(G_{X, \tilde r})$ such that $\nu_1\ne\nu_2$ and \eqref{e4.21} holds. It suffices to show that inequality \eqref{e4.22} holds for an infinite subsequence $\tilde r'$ of $\tilde r.$ Let $\tilde x^{1}=(x_{n}^{1})_{n\in\mathbb N}\in\nu_{1}$ and $\tilde x^{2}=(x_{n}^{2})_{n\in\mathbb N}\in\nu_{2}$ and let $p\in X.$ Since $\nu_{1}\ne\nu_{2}$ and $\rho^{0}(\nu_1)=\rho^{0}(\nu_2),$ we have
\begin{equation}\label{e4.24}
\lim_{n\to\infty}\frac{d(x_{n}^{1}, p)}{r_n}=\lim_{n\to\infty}\frac{d(x_{n}^{2}, p)}{r_n}>0
\end{equation}
and
\begin{equation*}
\quad \infty>\limsup_{n\to\infty}\frac{d(x_{n}^{1}, x_{n}^{2})}{r_n}>0.
\end{equation*}
Let $\mathbb{N}_e$ be an infinite subset of $\mathbb{N}$ such that $\mathbb{N} \setminus \mathbb{N}_e$ is also infinite and
\begin{equation}\label{e4.25}
\limsup_{n\to\infty} \frac{d(x_n^{1}, x_n^{2})}{r_n} = \lim_{\substack{n \to \infty\\ n \in \mathbb{N}_e}} \frac{d(x_{n}^{1}, x_{n}^{2})}{r_n}.
\end{equation}

To prove \eqref{e4.22} we consider a relation $\asymp$ on the set $2^{\mathbb N_{e}}$ of all subsets of $\mathbb N_{e}$ defined by the rule: $A\asymp B,$ if and only if the set $$A\bigtriangleup B=(A\setminus B)\cup (B\setminus A)$$ is finite, $|A\bigtriangleup B|<\infty.$ It is clear that $\asymp$ is reflexive and symmetric. Since for all $A, B, C\subseteq\mathbb N_{e}$ we have
$$
A\bigtriangleup C\subseteq (A\bigtriangleup B)\cup (B\bigtriangleup C),
$$
the relation $\asymp$ is transitive. Thus $\asymp$ is an equivalence on $2^{\mathbb N_{e}}$. If $A\subseteq\mathbb N_{e}$, then for every $B\subseteq\mathbb N_{e}$ we have
\begin{equation}\label{e4.26}
B=(B\setminus A)\cup (A\setminus (A\setminus B)).
\end{equation}
For every $A\subseteq\mathbb N_{e}$ write
$$
[A] := \{B\subseteq\mathbb N_{e}: B\asymp A\}.
$$
The set of all finite subsets of $\mathbb N_{e}$ is countable. Consequently equality \eqref{e4.26} implies $\bigl|[A]\bigr| = \aleph_0$ for every $A\subseteq\mathbb N_{e}$. Hence we have
\begin{equation}\label{e4.27}
\bigl|\{[A]: A\subseteq\mathbb N_{e}\}\bigr| = \bigl|2^{\mathbb N_{e}}\bigr| = \mathfrak{c}.
\end{equation}
Let $\mathbf{\mathcal{N}}\subseteq 2^{\mathbb N_{e}}$ be a set such that:

$\bullet$ For every $A\subseteq\mathbb N_{e}$ there is $N\in\mathbf{\mathcal{N}}$ with $A\asymp N$ holds;

$\bullet$ The implication
\begin{equation}\label{e4.28}
(N_1\asymp N_2)\Rightarrow (N_1 = N_2)
\end{equation}
holds for all $N_1, N_2\in\mathbf{\mathcal{N}}$.

It follows from \eqref{e4.27} that $|\mathbf{\mathcal{N}}|=\mathfrak{c}$. For every $N \in \mathcal{N}$ define the sequence $\tilde x(N) = (x_{n}(N))_{n\in\mathbb N}$ as
\begin{equation}\label{e4.29}
x_{n}(N):=
\begin{cases}
        x_{n}^{1}& \mbox{if} $ $ n\in N\\
        x_{n}^{2}& \mbox{if}$ $ n\in\mathbb N\setminus N,\\
         \end{cases}
\end{equation}
where $(x_{n}^{1})_{n\in\mathbb N},$ $(x_{n}^{2})_{n\in\mathbb N}\in Seq(X, \tilde r)$ satisfy \eqref{e4.24} and \eqref{e4.25}. It follows from \eqref{e4.24} and \eqref{e4.25} that
$$\lim_{n\to\infty}\frac{d(x_{n}(N), p)}{r_n}=\tilde{\tilde d}_{\tilde r}(\tilde x^{1})=\tilde{\tilde d}_{\tilde r}(\tilde x^{2})$$ for every $N\in\mathbf{\mathcal{N}}.$ Thus $\tilde x(N)\in Seq(\tilde X, \tilde r).$ If $N_1$ and $N_2$ are distinct elements of $\mathbf{\mathcal{N}},$ then the equality
$$d(x_{n}(N_1), x_{n}(N_2))=d(x_{n}^{1}, x_{n}^{2})$$ holds for every $n\in N_{1}\bigtriangleup N_{2}.$
Using \eqref{e4.25} and the definition of $\asymp$ we see that the set $N_{1}\bigtriangleup N_{2}$ is infinite for all distinct $N_{1}, N_{2}\in\mathbf{\mathcal{N}}.$ Consequently, we have
\begin{equation}\label{e4.30}
\limsup_{n\to\infty}\frac{d(x_{n}(N_1), x_{n}(N_2))}{r_n}>0\quad\text{and}\quad\liminf_{n\to\infty}\frac{d(x_{n}(N_1), x_{n}(N_2))}{r_n}=0.
\end{equation}
For every $N\in\mathbf{\mathcal{N}}$ write
\begin{equation}\label{e4.31}
\nu_{N}=\{\tilde x\in Seq(X, \tilde r): \tilde d_{\tilde r}(\tilde x, \tilde x(N))=0\}.
\end{equation}
The first inequality in \eqref{e4.30} implies that $\nu_{N_1}\ne\nu_{N_2}$ if $N_1\ne N_2.$ From $|\mathcal{N}| = \mathfrak{c}$, we obtain
$$
|\{\nu_{N}: N \in \mathcal{N}\}| = \mathfrak{c}.
$$
Inequality \eqref{e4.22} follows.
\end{proof}

\begin{remark}\label{r4.9}
The existence of continuum many sets $A_\gamma \subseteq  \mathbb{N}$ satisfying, for all  distinct $\gamma_1$ and $\gamma_2$, the equalities
$$
|A_{\gamma_1}\setminus A_{\gamma_2}| = |A_{\gamma_2}\setminus A_{\gamma_1}| = |A_{\gamma_1}\cap A_{\gamma_2}| = \aleph_0
$$
are well know. See, for example, Problem 41 of Chapter 4 in~\cite{KT}.
\end{remark}

\begin{corollary}\label{c4.8}
Let $(X, d)$ be an unbounded metric space, $n\ge 2$ be an integer number such that $\left|\Omega_{\infty, \tilde r}^{X}\right|\le n$ holds for every pretangent space $\Omega_{\infty, \tilde r}^{X}$. Then, for every $\tilde r$, we have
$$
\text{either } \left|\mathbf{\Omega_{\infty, \tilde r}^{X}}\right|\le 2^{n-1} \text{ or } \left|\mathbf{\Omega_{\infty, \tilde r}^{X}}\right|\ge \mathfrak{c},
$$
where $\left|\mathbf{\Omega_{\infty, \tilde r}^{X}}\right|$ is the cardinal number of distinct pretangent spaces to $(X,d)$ at infinity with respect to $\tilde r$.
\end{corollary}

\begin{proof}
Since every pretangent space $\Omega_{\infty, \tilde r}^{X}$ is a subset of $V(G_{X, \tilde r}),$ inequality \eqref{e4.20} implies that the number of pretangent is less than or equal to $2^{n-1}$ if $\rho^{0}: V(G_{X, \tilde r})\to \mathbb{R}$ is injective. Otherwise \eqref{e4.30} implies that distinct $\nu_{N_1}$ and $\nu_{N_2}$ defined by \eqref{e4.31} belong to the distinct pretangent spaces $^{1}\Omega_{\infty, \tilde r}^{X}$ and $^{2}\Omega_{\infty, \tilde r}^{X}$.
\end{proof}

Recall that a graph $G=(V, E)$ is trivial if $|V|=1.$ Moreover, if $|V|=2$ and $G$ is connected, then $G$ is called a $1$-path (see Figure~1).

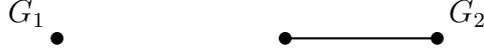
\begin{figure}[h]
\begin{center}
\begin{tikzpicture}[scale=1,thick]
\coordinate [label=above left:$G_1$] (A) at (0,0);
\coordinate (B) at (3,0);
\coordinate [label=above right:$G_2$] (C) at (5,0);
\draw (B) -- (C);
\draw [fill=black, draw=black] (A) circle (2pt);
\draw [fill=black, draw=black] (B) circle (2pt);
\draw [fill=black, draw=black] (C) circle (2pt);
\end{tikzpicture}
\end{center}
\caption{$G_1$ is trivial and $G_2$ is an 1-path.}
\label{ris1}
\end{figure}

Theorem~\ref{t4.5} and Theorem~\ref{t4.3} imply the next result for $n=2$.

\begin{corollary}\label{c4.9}
Let $(X, d)$ be an unbounded metric space. Then the following statements are equivalent.
\begin{enumerate}
\item[\rm(i)] The inequality
\begin{equation}\label{e4.2}
\left|\Omega_{\infty, \tilde r}^{X}\right|\le 2
\end{equation}
holds for every $\Omega_{\infty, \tilde r}^{X}$.
\item[\rm(ii)] The limit relation
\begin{equation}\label{e4.3}
\lim_{x,y\to\infty} F_2 (x, y)=0
\end{equation}
holds.
\item[\rm(iii)] For every scaling sequence $\tilde r$, the net $G_{X, \tilde r}$ of pretangent spaces is either trivial or this is a $1$-path.
\end{enumerate}
\end{corollary}

\begin{proof}
It suffices to note that the function $\rho^0\colon V(G_{X, \tilde{r}}) \to \mathbb{R}$ is injective for every $\tilde{r}$ if (i) holds. Indeed, if $\tilde{r}$ is a scaling sequence such that $\rho^0(\nu_1) = \rho^0(\nu_2)$ and $\nu_1$, $\nu_2$ are distinct vertices of $G_{X, \tilde{r}}$, then $\rho^0(\nu_1) = \rho^0(\nu_2) >0$. Consequently we can find a subsequence $\tilde{r}'$ of $\tilde{r}$ and $\Omega_{\infty, \tilde r'}^X$ such that $\nu_0$, $\nu_1$, $\nu_2 \in \Omega_{\infty, \tilde r'}^X$. Hence the inequality $|\Omega_{\infty, \tilde r}^X| \geq 3$ holds which contradicts (i).
\end{proof}

\begin{corollary}\label{c4.10}
Let $(X, d)$ be an unbounded metric space such that every pretangent space to $(X, d)$
at infinity contains at most two points. Then, for every scaling sequence $\tilde r,$ there is a unique pretangent spaces $\Omega_{\infty, \tilde r}^{X}.$
\end{corollary}
\begin{proof}
Suppose contrary that $^{1}\Omega_{\infty, \tilde r}^{X}$ and $^{2}\Omega_{\infty, \tilde r}^{X}$ are distinct pretangent spaces to $X$ with the same scaling sequence $\tilde r.$ Let $^{1}\tilde X_{\infty, \tilde r}$ and $^{2}\tilde X_{\infty, \tilde r}$ be maximal self-stable subsets of $Seq(X, \tilde r)$ such that $^{i}\Omega_{\infty, \tilde r}^{X}$ is the metric identification of
$^{i}\tilde X_{\infty, \tilde r}$, $i=1,2, \ldots \, .$ Since $^{1}\Omega_{\infty, \tilde r}^{X}\ne^{2}\Omega_{\infty, \tilde r}^{X}$ we have also $^{1}\tilde X_{\infty, \tilde r}\ne^{2}\tilde X_{\infty, \tilde r}.$ It implies
\begin{equation}\label{e4.12}
^{1}\tilde X_{\infty, \tilde r}\setminus^{2}\tilde X_{\infty, \tilde r}\ne\varnothing\ne^{2}\tilde X_{\infty, \tilde r}\setminus^{1}\tilde X_{\infty, \tilde r}
\end{equation}
because $^{1}\tilde X_{\infty, \tilde r}$ and $^{2}\tilde X_{\infty, \tilde r}$ are \emph{maximal} self-stable. Using \eqref{e4.12} we obtain
\begin{equation*}
^{1}\Omega_{\infty, \tilde r}^{X}\setminus^{2}\Omega_{\infty, \tilde r}^{X}\ne^{2}\Omega_{\infty, \tilde r}^{X}\setminus^{1}\Omega_{\infty, \tilde r}^{X}.
\end{equation*}
Moreover, $\tilde X_{\infty, \tilde r}^{0}\in^{1}\Omega_{\infty, \tilde r}^{X}\cap^{2}\Omega_{\infty, \tilde r}^{X}.$ Consequently, we have
\begin{equation*}
\left|\bar{\Omega}_{\infty, \tilde r}^{X}\right|\ge\left|^{1}\Omega_{\infty, \tilde r}^{X}\cap^{1}\Omega_{\infty, \tilde r}^{X}\right|+\left|^{1}\Omega_{\infty, \tilde r}^{X}\setminus^{2}\Omega_{\infty, \tilde r}^{X}\right|+\left|^{2}\Omega_{\infty, \tilde r}^{X}\setminus^{1}\Omega_{\infty, \tilde r}^{X}\right|\ge 3,
\end{equation*}
contrary to statement (iii) of Theorem~\ref{c4.9}.
\end{proof}

The following example shows that, for any $n \geq 3$, the equality
$$
\lim_{x_1, \ldots, x_n\to\infty} F_n(x_1, \ldots, x_n) = 0
$$
is not sufficient for the finiteness of the net $G_{X, \tilde{r}}$. In what follows $\mathbb{C}$ is the complex plane.

\begin{example}\label{ex4.10}
Let $\tilde{r} = (r_m)_{m \in \mathbb{N}}$ be a strictly increasing sequence of positive real numbers such that
$$
\lim_{m\to\infty} \frac{r_{m+1}}{r_m} = \infty,
$$
let $n \geq 2$ be an integer number, let $R_i = \{z \in \mathbb{C}\colon \arg{z} = \theta_i\}$ be the rays starting at origin with the angles of $\theta_i = \frac{\pi i}{2n}$ with the positive real axis, $i=0$, $\ldots$, $n-1$ and let
$$
C_m = \{z \in \mathbb{C}\colon |z|=r_m\}
$$
be the circles in $\mathbb{C}$ with radius $r_m$, $m \in \mathbb{N}$, and the center $0$. Write
$$
X_n:= \left(\bigcup_{i=0}^{n-1} R_i\right) \bigcap \left(\bigcup_{m=1}^{\infty} C_m\right)
$$
and define the distance function $d$ on $X_n$ as
$$
d(z,w) = |z-w|
$$
(see Figure~\ref{ris2} for $X_n$ with $n=3$).

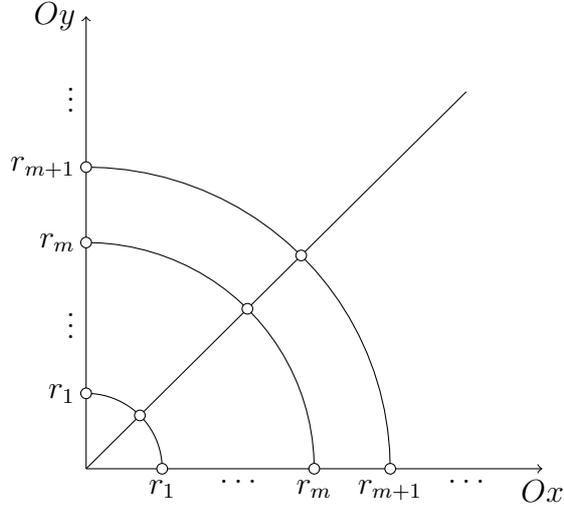
\begin{figure}[h]
\begin{center}
\begin{tikzpicture}[scale=1]
\draw[->] (0,0) -- (6,0) node[below] {$Ox$};
\draw[->] (0,0) -- (0,6) node[left] {$Oy$};
\draw[domain=0:pi/2] plot ({cos(\x r)}, {sin(\x r)});
\draw[domain=0:pi/2] plot ({3*cos(\x r)}, {3*sin(\x r)});
\draw[domain=0:pi/2] plot ({4*cos(\x r)}, {4*sin(\x r)});
\draw[domain=0:5] plot (\x,\x);
\draw [fill=white] (1,0) circle(2pt) node[below] {$r_1$};
\node [below] at (2,0) {$\ldots$};
\draw [fill=white] (3,0) circle(2pt) node[below] {$r_m$};
\draw [fill=white] (4,0) circle(2pt) node[below] {$r_{m+1}$};
\node [below] at (5,0) {$\ldots$};
\draw [fill=white] (0,1) circle(2pt) node[left] {$r_1$};
\node [left] at (0,2) {$\vdots$};
\draw [fill=white] (0,3) circle(2pt) node[left] {$r_m$};
\draw [fill=white] (0,4) circle(2pt) node[left] {$r_{m+1}$};
\node [left] at (0,5) {$\vdots$};
\draw [fill=white] ({cos(pi/4 r)},{sin(pi/4 r)}) circle(2pt);
\draw [fill=white] ({3*cos(pi/4 r)},{3*sin(pi/4 r)}) circle(2pt);
\draw [fill=white] ({4*cos(pi/4 r)},{4*sin(pi/4 r)}) circle(2pt);
\end{tikzpicture}
\end{center}
\caption{The graphical representation of $X_3$. The points of $X_3$ are depicted here as small circles.}
\label{ris2}
\end{figure}

Then we obtain
$$
\lim_{\substack{x_1, \ldots, x_{n+1} \to \infty\\ x_1, \ldots, x_{n+1} \in X_n}} F_{n+1} (x_1, \ldots, x_{n+1}) = 0 < \limsup_{\substack{x_1, \ldots, x_n  \to \infty\\ x_1, \ldots, x_{n} \in X_n}} F_{n} (x_1, \ldots, x_n)
$$
and $|\mathbf{\Omega}_{\infty, \tilde r}^{X_n}|=\mathfrak{c}$. In particular, for $n=3$, the equality
$$
\limsup_{\substack{x, y, z\to\infty\\ x, y, z \in X_3}} F_{3} (x, y, z) = 2\sqrt{2} - 2
$$
holds. (See Figure~\ref{ris3} for all possible pretangent spaces to $X_3$ at infinity with respect to $\tilde{r}$.)

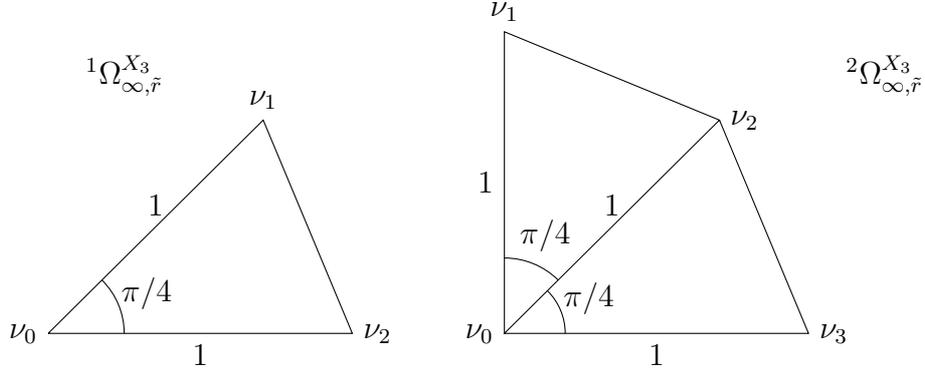
\begin{figure}[h]
\begin{center}
\begin{tikzpicture}[scale=1]
\coordinate [label=left:$\nu_0$] (A) at (0,0);
\coordinate [label=above:$\nu_1$] (B) at ($(45:4)+(A)$);
\coordinate [label=right:$\nu_2$] (C) at ($(0:4)+(A)$);
\draw (A)--(B) node[midway, above] {$1$} --(C)--(A) node[midway, below] {$1$};
\pic [draw, "$\pi/4$", angle eccentricity=1.4, angle radius=1cm] {angle=C--A--B};
\node [above] at ($(A)+(1,3)$) {$^1\Omega_{\infty, \tilde r}^{X_3}$};
\coordinate [label=left:$\nu_0$] (A1) at (6,0);
\coordinate [label=above:$\nu_1$] (B1) at ($(90:4)+(A1)$);
\coordinate [label=right:$\nu_2$] (C1) at ($(45:4)+(A1)$);
\coordinate [label=right:$\nu_3$] (D1) at ($(0:4)+(A1)$);
\draw (A1)--(B1) node[midway, left] {$1$} -- (C1) -- (A1) node[midway, above] {$1$} -- (D1) node[midway, below] {$1$} -- (C1);
\pic [draw, "$\pi/4$", angle eccentricity=1.4, angle radius=1cm] {angle=C1--A1--B1};
\pic [draw, "$\pi/4$", angle eccentricity=1.5, angle radius=.8cm] {angle=D1--A1--C1};
\node [above] at ($(D1)+(1,3)$) {$^2\Omega_{\infty, \tilde r}^{X_3}$};
\end{tikzpicture}
\end{center}
\caption{Every pretangent space $\Omega_{\infty, \tilde r}^{X_3}$ (with respect to $\tilde{r}$ given above) is isometric either $^1\Omega_{\infty, \tilde r}^{X_3}$ or $^2\Omega_{\infty, \tilde r}^{X_3}$.}
\label{ris3}
\end{figure}
\end{example}

It can be shown that for every finite metric space $Y$ there is an unbounded metric space $X$ and a scaling sequence $\tilde r$ such that $Y$ is isometric to a pretangent space $\Omega_{\infty, \tilde r}^{X}.$ We will consider the case when $Y$ is \emph{strongly rigid}.

\begin{definition}\label{D4.12}
A metric space $(Y, \delta)$ is said to be strongly rigid if for all $x, y, z, w\in Y$ the conditions $$\delta(x, y)=\delta(w, z)\quad\text{and}\quad x\ne y$$ imply that $\{x, y\}=\{z, w\}.$ \end{definition}

\begin{remark}
See \cite{J} and \cite{M} for some interesting properties of strongly rigid metric spaces.
\end{remark}

\begin{example}\label{ex4.13}
Let $(Y, \delta)$ be a finite, nonempty and strongly rigid metric space, $Y=\{y_1, ..., y_k\}$ and let $\tilde r=(r_j)_{j\in\mathbb N}$ be a scaling sequence such that
\begin{equation}\label{e4.32}
\lim_{j\to\infty}\frac{r_{j+1}}{r_j}=\infty.
\end{equation}
\end{example}
We will define a metric space $(X, d)$ as a subset of the finite dimensional normed space $l_{k}^{\infty}$ of $k$-tuples $x=(x_1, ..., x_k)$ of real numbers with the sup norm $$\|x\|_{\infty}=\sup_{1\le i\le k}|x_i|.$$
Let $y^{*}$ be a fixed point of $Y.$ The Kuratowski embedding $s: Y\to l_{k}^{\infty}$ can be defined as
\begin{equation}\label{e4.33}
s(y)=\begin{pmatrix}
\delta(y, y_1)-\delta(y_1, y^{*})\\
..........................\\
\delta(y, y_k)-\delta(y_k, y^{*})\\
\end{pmatrix}.
\end{equation}
Write
\begin{equation}\label{e4.34}
X:=\bigcup_{n\in\mathbb N}r_{n}s(Y),
\end{equation}
where $r_{n}s(Y):=\{r_{n}s(y): y\in Y\}$ and define $d(x, y):=\parallel x-y\parallel_{\infty}$ for all $x, y\in X.$ It follows directly from \eqref{e4.33} that
\begin{equation*}
s(y^{*})=\begin{pmatrix}
0\\
.....\\
0\\
\end{pmatrix}.
\end{equation*}
For convenience we can suppose that the distinguished point $p$ of $X$ is $s(y^{*}),$
\begin{equation*}
p=\begin{pmatrix}
0\\
.....\\
0\\
\end{pmatrix}.
\end{equation*}
Hence, for every $x\in X,$ we have
\begin{equation}\label{e4.35}
d(x, p)=\parallel x\parallel_{\infty}.
\end{equation}
Suppose now
$$\tilde x=(x_j)_{j\in\mathbb N}\in Seq (X, \tilde r)$$ such that
\begin{equation}\label{e4.36}
\tilde{\tilde d}_{\tilde r}(\tilde x)=\lim_{j\to\infty}\frac{d(x_j, p)}{r_j}=\lim_{j\to\infty}\frac{\parallel x_{j}\parallel_{\infty}}{r_j}>0.
\end{equation}
By \eqref{e4.34}, for $j\in\mathbb N,$ there are $n\in\mathbb N$ and $y\in Y$ satisfying the equality $x_{j}=r_{n}s(y).$ It is well known that the Kuratowski embeddings are isometric (see, for example, \cite[the proof of Theorem III.8.1]{B}). Consequently
\begin{equation}\label{e4.36*}
\frac{1}{r_j}\parallel x_j\parallel_{\infty}=\frac{r_n}{r_j}\parallel s(y)\parallel_{\infty}=\frac{r_n}{r_j}\delta(y, y^{*}).
\end{equation}
Now \eqref{e4.36} implies $y\ne y^{*}$ for all sufficiently large $j.$ Moreover, using \eqref{e4.32} and \eqref{e4.36*} we obtain $r_n=r_j$ for all sufficiently large $j$ which is an equivalent for $n=j$ for all sufficiently large $j.$ Hence, if $\tilde x=(x_j)_{j\in\mathbb N}$ belongs $Seq(X, \tilde r)$ and $\tilde{\tilde d}_{\tilde r}(x)>0,$ then for every sufficiently large $j$ there is $y(j)\in\mathbb N$ such that
\begin{equation*}
\frac{1}{r_j}\parallel x_j\parallel_{\infty}=\delta(y(j), y^{*}).
\end{equation*}
Since $(Y, \delta)$ is strongly rigid, the limit $\mathop{\lim}\limits_{j\to\infty}\delta(y(j), y^{*})$ exists if and only if there is $y'\in Y$ such that $$y(j)=y'$$ holds for all sufficiently large $j.$ This result and statement (vi) of Proposition~\ref{pr2.2} imply that every two $\tilde x, \tilde y\in Seq(X, \tilde r)$ are mutually stable and
\begin{equation*}
\lim_{j\to\infty}\frac{d(x_j, y_j)}{r_j}=\parallel s(x')- s(y')\parallel_{\infty}=\delta(x', y'),
\end{equation*}
where $x'$ and $y'$ are the points of $Y$ for which
\begin{equation*}
y(j)=y'\quad\text{and}\quad x(j)=x'
\end{equation*}
hold for all sufficiently large $j.$ It is clear that, for every $y\in Y,$ we have
\begin{equation*}
(r_{j}s(y))_{j\in\mathbb N}\in Seq(X, \tilde r).
\end{equation*}
Thus there is the unique pretangent space $\Omega_{\infty, \tilde r}^{X}$ and this space is isometric to $(Y, \delta).$

Analyzing the construction of Example~\ref{ex4.13} we obtain the following proposition.

\begin{proposition}\label{pr4.15}
Let $(Y, \delta)$ be a finite nonempty metric space containing a point $y^{*}$ for which
$$\delta(y^{*}, x)\ne\delta(y^{*}, z)$$ whenever $x$ and $z$ are distinct points of $Y.$ Then there are an unbounded metric space $(X, d)$ and a scaling sequence $\tilde r$ such that:
\begin{enumerate}
\item[\rm(i)] There is a unique pretangent space $\Omega_{\infty, \tilde r}^{X};$
\item[\rm(ii)] The pretangent space $\Omega_{\infty, \tilde r}^{X}$ is tangent;
 \item[\rm(iii)]  There is an isometry $f: \Omega_{\infty, \tilde r}^{X}\to Y$ such that $f(\nu_0)=y^{*},$ where $\nu_{0}=\tilde X_{\infty, \tilde r}^{0}$ is the distinguished point of $\Omega_{\infty, \tilde r}^{X}.$
\end{enumerate}
\end{proposition}
For the proof note that statement (ii) can be obtained by application of Corollary~\ref{c4.4} and Theorem~\ref{t4.3} to $(X, d)$ constructed in Example~\ref{ex4.13}.

A simple modification of Example~\ref{ex4.13} gives us the following result.

\begin{theorem}
For every finite nonempty metric space $(Y, \delta)$ and every $y^* \in Y$ there are an unbounded metric space $(X,d)$, and a scaling sequence $\tilde{r}$ and an isometry $f\colon \Omega_{\infty, \tilde r}^X \to Y$ such that $f(\nu_0) = y^*$ holds and $\Omega_{\infty, \tilde r}^X$ is tangent.
\end{theorem}

The last theorem does not have direct generalization to the case of infinite $(Y, \delta)$ even if $(Y, \delta)$ is complete, separable and strongly rigid. (See Example~\ref{ex5.10} in the last section of the paper.)

\section{Finite tangent spaces and strong porosity at a point}

Theorem~\ref{t4.3} gives, in particular, a condition guaranteeing the finiteness of all pretangent spaces. The goal of present section is to obtain the \emph{existence conditions} for finite tangent spaces.

Let $(X, d)$ be an unbounded metric space and let $p\in X.$
The finiteness of $\Omega_{\infty, \tilde r}^{X}$ is closely connected with a porosity of the set $$Sp(X):=\{d(x, p): x\in X\}$$ at infinity.
\begin{definition}\label{d5.1}
Let $E\subseteq\mathbb R^{+}$. The \emph{porosity} of $E$ at infinity is the quantity
\begin{equation}\label{e5.1}
p^{+}(E, \infty):=\limsup_{h\to\infty}\frac{l(\infty, h, E)}{h},
\end{equation}
where $l(\infty, h, E)$ is the length of the longest interval in the set $[0, h]\setminus E$. The set $E$ is \emph{strongly porous at infinity} if $p^{+}(E, \infty)=1$ and, respectively, $E$ is \emph{nonporous at infinity} if $p^{+}(E, \infty)=0$.
\end{definition}

The standard definition of porosity at a finite point can be found in \cite{Th}. See \cite{ADK, BD1, BD2, BD3} for some applications of porosity to studies of pretangent spaces  at finite points of metric spaces.

\begin{lemma}\label{l5.2}
Let $E\subseteq\mathbb R^{+}$ and let $p_1 \in (0,1)$. If the double inequality
\begin{equation}\label{e5.2}
p^{+}(E, \infty)<p_1<1
\end{equation}
holds, then for every infinite, strictly increasing sequence of real numbers $r_n$ with $\mathop{\lim}\limits_{n\to\infty}r_n=\infty$ there is a subsequence $(r_{n_k})_{k\in\mathbb N}$ such that for every $k\in\mathbb N$ there are points $x_{1}^{(k)}, \ldots, x_{k}^{(k)}\in E$ which satisfy the inequalities
\begin{equation}\label{e5.3}
\begin{aligned}
r_{n_k} < \frac{r_{n_k}}{1-p_1} & \le x_{1}^{(k)}\le \frac{r_{n_k}}{(1-p_1)^{2}}\\
<\frac{r_{n_k}}{(1-p_1)^{3}} & \le x_{2}^{(k)}\le\frac{r_{n_k}}{(1-p_1)^{4}}\\
\ldots & \ldots \ldots \ldots \ldots\\
<\frac{r_{n_k}}{(1-p_1)^{2k-1}} & \le x_{k}^{(k)} \le \frac{r_{n_k}}{(1-p_1)^{2k}}<r_{n_{k+1}}.
\end{aligned}
\end{equation}
\end{lemma}
\begin{proof}
Suppose that \eqref{e5.2} holds. Let $n_1$ be the first natural number such that $l(\infty, h, E)< p_{1}h$ for all $h\in(r_{n_1}, \infty).$ If $r_{n_1}, \ldots, r_{n_k}$ are defined, then
write $n_{k+1}$ for the first $m$ with
$$
r_m>(1-p_1)^{-2k} r_{n_k}.
$$
It is easy to show that the equality
$$
\frac{r_{n_k}}{(1-p_1)^{m}}-\frac{r_{n_k}}{(1-p_1)^{m-1}} = p_1 \cdot \frac{r_{n_k}}{(1-p_1)^{m}}
$$
holds for all $m\in\mathbb N$. Using the last equality, Definition~\ref{d5.1} and inequality \eqref{e5.2} we obtain
$$
E\cap\left[\frac{r_{n_k}}{(1-p_1)^{m}}, \frac{r_{n_k}}{(1-p_1)^{m+1}}\right]\ne\varnothing
$$
for all $m\in\mathbb N.$ Hence, there are points $x_{1}^{(k)}, \ldots, x_{k}^{(k)}\in E$ which satisfy \eqref{e5.3}.
\end{proof}

\begin{theorem}\label{th5.3}
Let $(X, d)$ be an unbounded metric space, $p\in X.$ The following statements are equivalent:
\begin{enumerate}
\item[\rm(i)] The set $Sp(X)$ is strongly porous at infinity;
\item[\rm(ii)] There is a single-point, tangent space $\Omega_{\infty, \tilde r}^{X}$;
\item[\rm(iii)] There is a finite tangent space $\Omega_{\infty, \tilde r}^{X}$;
\item[\rm(iv)]  There is a compact tangent space $\Omega_{\infty, \tilde r}^{X};$
\item[\rm(v)]  There is a bounded, separable tangent space $\Omega_{\infty, \tilde r}^{X}$.
\end{enumerate}
\end{theorem}
\begin{proof}
(i)$\Rightarrow$(ii) Suppose the equality
\begin{equation*}
p^{+}(Sp (X), \infty)=1
\end{equation*}
holds. Let $\tilde h=(h_n)_{n\in\mathbb N}$ be a strictly increasing sequence of positive numbers such that
$$
\lim_{n\to\infty} h_n = \infty\quad\mbox{and}\quad \lim_{n\to\infty}\frac{l(\infty, h_n, Sp(X))}{h_n}=1.
$$
Let us consider a sequence of intervals $(c_n, d_n)\subseteq[0, h_n]\setminus Sp (X)$ for which
\begin{equation}\label{e5.4}
\lim_{n\to\infty}\frac{|d_n - c_n|}{h_n}=1.
\end{equation}
Passing, if it is necessary, to a subsequence we suppose that
\begin{equation}\label{e5.5}
0<c_n<d_n\leq h_n
\end{equation}
holds for every $n\in\mathbb N$. A pretangent $\Omega_{\infty,\tilde r}^X$ is single-point if and only if $\tilde{X}_{\infty,\tilde r} = \tilde{X}_{\infty,\tilde r}^0$ holds for the corresponding $\tilde{X}_{\infty,\tilde r}$. Hence, it suffices to prove that
\begin{equation}\label{e5.7}
\lim_{n\to\infty}\frac{d(x_n, p)}{r_n}=0
\end{equation}
holds for every $x \in \tilde{X}_{\infty,\tilde r}$. Write
\begin{equation}\label{e5.6}
r_n:=\sqrt{d_{n}c_{n}}
\end{equation}
for every $n \in \mathbb{N}$ and define $\tilde{r} := (r_n)_{n \in \mathbb{N}}$.

Let us prove equality \eqref{e5.7}. It is evident that \eqref{e5.4} and \eqref{e5.5} imply the limit relations
\begin{equation}\label{e5.8}
\lim_{n\to\infty}\frac{c_n}{h_n}=0 \quad\mbox{and}\quad \lim_{n\to\infty}\frac{d_n}{h_n}=1.
\end{equation}
Consequently, we obtain
\begin{equation}\label{e5.9}
\lim_{n\to\infty}\frac{c_n}{d_n}=0 \quad\mbox{and} \quad \lim_{n\to\infty}\frac{d_n}{c_n}=\infty.
\end{equation}
Since $(c_n, d_n)\subseteq[0, h_n]\setminus Sp(X)$, we have either $d(x_n, p)\le c_n$ or $d(x_n p)\ge d_n$ for all $n\in\mathbb N.$ Thus, either
\begin{equation}\label{e5.10}
\frac{d(x_n, p)}{r_n}\le\sqrt{\frac{c_n}{d_n}}
\end{equation} or
\begin{equation}\label{e5.11}
\frac{d(x_n, p)}{r_n}\ge\sqrt{\frac{d_n}{c_n}}
\end{equation}
holds
for every $n\in\mathbb N$. The second relation in \eqref{e5.9} implies that \eqref{e5.11} cannot be valid for sufficient large $n$ because $\tilde{\tilde d}_{\tilde{r}} (\tilde{x})$ is finite. Now, \eqref{e5.7} follows from~\eqref{e5.10}.

It is proved that if $\tilde r$ is defined by \eqref{e5.6}, then there is a unique pretangent space $\Omega_{\infty, \tilde r}^{X}$ and this space is single-point. Note also that $\Omega_{\infty, \tilde r}^{X}$ is tangent. To prove it we can consider the subsequences $\tilde x'=(x_{n_k})_{k\in\mathbb N},$ $\tilde z'=(z_{n_k})_{k\in\mathbb N}$ and $\tilde r'=(r_{n_k})_{k\in\mathbb N}$ of $\tilde x$, $\tilde z$ and $\tilde r$, and repeat the proof of equality \eqref{e5.7} substituting $d_{n_k}, c_{n_k}, h_{n_k}$ and $r_{n_k}$ instead of $d_n, c_n, h_n$ and $r_n,$ respectively. The implication (i)$\Rightarrow$(ii) follows.

(ii)$\Rightarrow$(iii), (iii)$\Rightarrow$(iv), (iv)$\Rightarrow$(v) The implications are evident.

(v)$\Rightarrow$(i) Suppose statement (v) holds but there is $p_{1}\in (0, 1)$ such that $p^{+}(Sp(X), \infty) < p_1$. Let $\tilde r=(r_n)_{n\in\mathbb N}$ be a scaling sequence and let $\Omega_{\infty, \tilde r}^{X}$ be bounded, separable and tangent. Applying Lemma~\ref{l5.2} with $E = Sp(X)$ we can find a subsequence $\tilde r' = (r_{n_k})_{k\in\mathbb N}$ of the sequence $\tilde r$ such that for every $k\in\mathbb N$ there are points $x_{1}^{(k)}, \ldots, x_{k}^{(k)}\in X$ for which
\begin{equation}\label{e5.12}
\begin{aligned}
\frac{1}{1-p_1} & \le\frac{d(x_{1}^{(k)}, p)}{r_{n_{k}}}\le\frac{1}{(1-p_1)^{2}},\\
\frac{1}{(1-p_1)^{3}} & \le\frac{d(x_{2}^{(k)}, p)}{r_{n_{k}}}\le\frac{1}{(1-p_1)^{4}},\\
\ldots & \ldots\ldots\ldots \ldots \ldots \ldots \ldots\\
\frac{1}{(1-p_1)^{2k-1}} & \le\frac{d(x_{k}^{(k)}, p)}{r_{n_{k}}}\le\frac{1}{(1-p_1)^{2k}}.
\end{aligned}
\end{equation}
Let $\tilde z=(z_k)_{k\in\mathbb N}\in\tilde{\tilde X}_{\infty, \tilde q}$ and $\tilde q=(q_k)_{k\in\mathbb N}.$
Write $\tilde x_{j}$ for the $j$-th column of the following infinite matrix

\begin{equation*}
\begin{pmatrix}
x_{1}^{(1)}&z_1&z_1&z_1&z_1& \ldots\\
x_{1}^{(2)}&x_{2}^{(2)}&z_2&z_2&z_2& \ldots\\
x_{1}^{(3)}&x_{2}^{(3)}&x_{3}^{(3)}&z_3&z_3& \ldots\\
x_{1}^{(4)}&x_{2}^{(4)}&x_{3}^{(4)}&x_{4}^{(4)}&z_4& \ldots\\
\ldots & \ldots & \ldots &\ldots & \ldots & \ldots
\end{pmatrix}.
\end{equation*}
It follows from \eqref{e5.12} that the inequalities
\begin{equation}\label{e5.13}
\frac{1}{(1-p_1)^{2j-1}}\le\liminf_{k\to\infty}\frac{d(x_{j}^{(k)}, p)}{r_{n_k}}\le\limsup_{k\to\infty}\frac{d(x_{j}^{(k)}, p)}{r_{n_k}}\le\frac{1}{(1-p_1)^{2j}}
\end{equation}
holds for all $j\in\mathbb N$. Let $\tilde X_{\infty, \tilde r}$ be the maximal self-stable family with the metric identification $\Omega_{\infty, \tilde r}^{X}$ and let $\tilde X'_{\infty, \tilde r}=\{(x_{n_k})_{k\in\mathbb N}: (x_n)_{n\in\mathbb N}\in\tilde X_{\infty, \tilde r}\}.$ The family $\mathfrak B:=\{\tilde x_{1}, \tilde x_{2}, \ldots\}$ satisfies the conditions of Lemma~\ref{l4.2}. Hence there is a subsequence $\tilde r''$ of $\tilde r'$ such that $\tilde X_{\infty, \tilde r''} \supseteq \mathfrak B'$. The first inequality in \eqref{e5.12} implies that $\Omega_{\infty, r''}^{X}$ is unbounded, contrary to (v).
\end{proof}

\section{Kuratowski limits of subsets of $\mathbb R$ and their applications to pretangent spaces}

Let $(Y, \delta)$ be a metric space. For any sequence $(A_n)_{n\in\mathbb N}$ of nonempty sets $A_n\subseteq Y$, the \emph{Kuratowski limit inferior} of $(A_n)_{n\in\mathbb N}$ is the subset $\mathop{Li}\limits_{n\to\infty} A_n$ of $Y$ defined by the rule:
\begin{equation}\label{e5.14}
\left(y\in\mathop{Li}\limits_{n\to\infty} A_n\right) \Leftrightarrow \left(\forall \varepsilon>0 \  \exists n_{0} \in \mathbb N\ \forall n \ge n_0\colon B(y, \varepsilon)\cap A_n \ne \varnothing \right),
\end{equation}
where $B(y, \varepsilon)$ is the open ball of radius $\varepsilon > 0$ centered at the point $y \in Y$,
$$
B(y, \varepsilon)=\{x\in Y: \delta(x, y) < \varepsilon\}.
$$
Similarly, the \emph{Kuratowski limit superior} of $(A_n)_{n\in\mathbb N}$ can be defined as the subset $\mathop{Ls}\limits_{n\to\infty}A_{n}$ of $Y$ for which

\begin{equation}\label{e5.15}
\left(y\in \mathop{Ls}\limits_{n\to\infty}A_{n}\right)\Leftrightarrow\left(\forall\varepsilon>0\,\, \forall n\in\mathbb N\,\, \exists n_{0}\ge n: B(y, \varepsilon)\cap A_{n_0}\ne\varnothing\right).
\end{equation}

\begin{remark}
The Kuratowski limit inferior and limit superior are basic concepts of set-valued analysis in metric spaces and have numerous applications (see, for example, \cite{AF}).
\end{remark}

We denote $tA := \{tx: x\in A\}$ for any nonempty set $A\subseteq\mathbb R$ and $t\in\mathbb R,$
and, $\nu_0 := \tilde X_{\infty, \tilde r}^{0} \in \Omega_{\infty, \tilde r}^{X}$ for any pretangent space $\Omega_{\infty, \tilde r}^{X}$ of an unbounded metric space $(X, d)$. Moreover, for every scaling sequence $\tilde{r}$, we denote by $\mathbf{\Omega_{\infty, \tilde r}^{X}}$ the set of all pretangent at infinity spaces to $(X,d)$ with respect to $\tilde{r}$. Write
\begin{equation}\label{e5.16}
Sp\left(\Omega_{\infty, \tilde r}^{X}\right) := \{\rho(\nu_0, \nu)\colon \nu \in \Omega_{\infty, \tilde r}^{X}\} \text{ and } Sp(X) := \{d(p,x)\colon x \in X\}.
\end{equation}

\begin{proposition}\label{p5.4}
Let $(X, d)$ be an unbounded metric space, $p\in X$, $\tilde r=(r_n)_{n\in\mathbb N}$ be a scaling sequence and let $\tilde{\mathbf{R}}$ be the set of all infinite subsequences of $\tilde r$. Then the equalities
\begin{equation}\label{e5.15*}
\bigcup_{\Omega_{\infty, \tilde r}^{X} \in \mathbf{\Omega_{\infty, \tilde r}^{X}}} Sp \left(\Omega_{\infty, \tilde r}^{X}\right) = \mathop{Li} \limits_{n\to\infty} \left(\frac{1}{r_n} Sp(X)\right),
\end{equation}
\begin{equation}\label{e5.16*}
\bigcup_{\Omega_{\infty, \tilde r'}^{X}\in\mathbf\Omega_{\infty, \tilde r'}^{X}, \,  \tilde r'\in\tilde{\mathbf{R}}}Sp(\Omega_{\infty, \tilde r'}^{X})=\mathop{Ls}\limits_{n\to\infty}\left(\frac{1}{r_n}(Sp(X))\right)
\end{equation}
hold.
\end{proposition}

\begin{proof}
Let us prove the inclusion
\begin{equation}\label{e5.17}
\bigcup_{\Omega_{\infty, \tilde r}^{X} \in \mathbf{\Omega_{\infty, \tilde r}^{X}}} Sp \left(\Omega_{\infty, \tilde r}^{X}\right) \subseteq \mathop{Li} \limits_{n\to\infty} \left(\frac{1}{r_n} Sp(X)\right).
\end{equation}
Let $\Omega_{\infty, \tilde r}^{X}\in\mathbf{\Omega_{\infty, \tilde r}^{X}}$ and $\nu \in \Omega_{\infty, \tilde r}^{X}$ be arbitrary. Let $\tilde X_{\infty, \tilde r}$ be a maximal self-stable family with the metric identification $\Omega_{\infty, \tilde r}^{X}$, and let $\tilde x=(x_n)_{n \in \mathbb N} \in \tilde X_{\infty, \tilde r}$, $\tilde z=(z_n)_{n\in\mathbb N} \in \tilde X_{\infty, \tilde r}^0$ such that
$$
\pi(\tilde x) = \nu \quad \mbox{and} \quad \pi(\tilde z)= \nu_0.
$$
Then, by the definition of pretangent spaces, we have
\begin{equation*}
\lim_{n\to\infty} \frac{d(x_n, z_n)}{r_n} = \rho(\nu_0, \nu).
\end{equation*}
Consequently, for every $\varepsilon>0$ the inequality
$$
\left|\frac{1}{r_n}d(x_n, p)-\rho(\nu_0, \nu)\right| < \varepsilon
$$
holds for all sufficiently large $n$. Since $\Omega_{\infty, \tilde r}^{X}$ is an arbitrary element of $\mathbf{\Omega_{\infty, \tilde r}^{X}}$ and $\nu$ is an arbitrary point of $\Omega_{\infty, \tilde r}^{X}$ and $\frac{1}{r_n} d(x_n, p) \in \frac{1}{r_n} Sp(X)$, inclusion \eqref{e5.17} follows.

To obtain the converse inclusion, we consider an arbitrary
\begin{equation}\label{e5.18}
t\in\mathop{Li} \limits_{n\to\infty} \left(\frac{1}{r_n} Sp(X)\right).
\end{equation}
It is evident that $0 \in Sp\left(\Omega_{\infty, \tilde r}^{X}\right)$ holds for every $\Omega_{\infty, \tilde r}^{X}$. Suppose $t>0$ and write
\begin{equation*}
\operatorname{dist}\left(t, \frac{1}{r_n}Sp(X)\right) := \inf\left\{|t-s|\colon s \in \frac{1}{r_n}Sp(X)\right\}.
\end{equation*}
Using \eqref{e5.14}, we see that \eqref{e5.18} holds if and only if
\begin{equation}\label{e5.19}
\lim_{n\to\infty} \textrm{dist}\left(t, \frac{1}{r_n} Sp(X)\right)=0.
\end{equation}
Consequently, there is a sequence $(\tau_n)_{n\in\mathbb N}$ such that
\begin{equation}\label{e5.20}
\lim_{n\to\infty}\left|\tau_n - t\right|=0
\end{equation}
and $\tau_n\in\frac{1}{r_n}Sp(X)$ for every $n\in\mathbb N$. Using the definition of $\frac{1}{r_n}Sp(X),$ we may rewrite the last statement as: ``There is a sequence $(x_n)_{n\in\mathbb N}\subset X$ such that
$$
\lim_{n\to\infty}\left|\frac{d(x_n, p)}{r_n}-t\right|=0
$$
holds''. Thus, we have
\begin{equation}\label{e5.21}
\lim_{n\to\infty}\frac{d(x_n, p)}{r_n}=t.
\end{equation}
The inequality $t>0$ implies that $(x_n)_{n\in\mathbb N}\in\tilde X_{\infty, \tilde r}$. Let $\tilde X_{\infty, \tilde r}$ be a maximal self-stable family for which $(x_n)_{n\in\mathbb N} \in \tilde X_{\infty, \tilde r}$ and let $\Omega_{\infty, \tilde r}^{X}$ be the metric identification of $\tilde X_{\infty, \tilde r}$. Limit relation \eqref{e5.21} implies $t \in Sp\left(\Omega_{\infty, \tilde r}^{X} \right)$. Since $t$ is an arbitrary positive number from $\mathop{Li} \limits_{n\to\infty} \left(\frac{1}{r_n} Sp(X)\right)$, we obtain
\begin{equation*}
\bigcup_{\Omega_{\infty, \tilde r}^{X}\in\mathbf{\Omega_{\infty, \tilde r}^{X}}}Sp\left(\Omega_{\infty, \tilde r}^{X}\right)\supseteq\mathop{Li}\limits_{n\to\infty}\left(\frac{1}{r_n}Sp(X)\right).
\end{equation*}
Equality \eqref{e5.15} follows.

Equality \eqref{e5.16*} follows from \eqref{e5.15} because, for every $t\ge 0,$ we have $t\in\mathop{Ls}\limits_{n\to\infty}\left(\frac{1}{r_n}Sp(X)\right)$ if and only if $t\in\mathop{Li}\limits_{n\to\infty}\left(\frac{1}{r_{n_k}}Sp(X)\right)$ holds for some $(r_{n_k})_{k\in\mathbb N}\in\tilde{\mathbf{R}}.$
\end{proof}

\begin{remark}\label{r5.5}
Let $\rho^0\colon V(G_{X, \tilde{r}}) \to \mathbb{R}^+$ be the function defined in Theorem~\ref{t4.5}. Then using equality \eqref{e2.10} we see that
$$
\rho^0\bigl(V(G_{X, \tilde{r}})\bigr) =\{\tilde{\tilde d}_{\tilde r}(\tilde x): \tilde x\in Seq (X, \tilde r)\}= \bigcup_{\Omega_{\infty, \tilde r}^{X} \in \mathbf{\Omega}_{\infty, \mathbf{\tilde{r}}}^{\mathbf{X}}} Sp\bigl(\Omega_{\infty, \tilde r}^{X}\bigr).
$$
\end{remark}

\begin{corollary}\label{c5.7}
Let $(X, d)$ be an unbounded metric space, $\tilde r$ be a scaling sequence and let $^{1}\Omega_{\infty, \tilde r}^{X}$ be tangent and separable. Then we have
\begin{equation}\label{e5.21*}
\mathop{Li}\limits_{n\to\infty}\left(\frac{1}{r_n}Sp(X)\right)=\mathop{Ls}\limits_{n\to\infty}\left(\frac{1}{r_n}Sp(X)\right)=
Sp\left(^{1}\Omega_{\infty, \tilde r}^{X}\right).
\end{equation}
\end{corollary}
\begin{proof}
Using Lemma~\ref{l4.2}, for every $\tilde r'\in\tilde{\mathbf{R}}$ and every $$s\in\mathop{\bigcup}\limits_{\Omega_{\infty, \tilde r'}^{X}\in\mathbf{\Omega}_{\infty, \tilde r'}^{X}}\Omega_{\infty, \tilde r'}^{X},$$ we can find $\nu$ belonging to the tangent space $^{1}\Omega_{\infty, \tilde r}^{X}$ such that $\rho(\nu_{0}, \nu)=s.$ Consequently
\begin{equation*}
Sp\left(^{1}\Omega_{\infty, \tilde r}^{X}\right)\supseteq \bigcup_{\Omega_{\infty, \tilde r'}^{X}\in\mathbf\Omega_{\infty, \tilde r'}^{X}, \,  \tilde r'\in\tilde{\mathbf{R}}}Sp(\Omega_{\infty, \tilde r'}^{X})
\end{equation*}
holds. It is evident that
\begin{equation*}
\bigcup_{\Omega_{\infty, \tilde r'}^{X}\in\mathbf\Omega_{\infty, \tilde r'}^{X}, \,  \tilde r'\in\tilde{\mathbf{R}}}Sp(\Omega_{\infty, \tilde r'}^{X})\supseteq\bigcup_{\Omega_{\infty, \tilde r}^{X}\in\mathbf{\Omega}_{\infty, \tilde r}^{X}} Sp(\Omega_{\infty, \tilde r}^{X})\supseteq Sp(^{1}\Omega_{\infty, \tilde r}^{X}).
\end{equation*}
Hence we have the equalities
\begin{equation*}
\bigcup_{\Omega_{\infty, \tilde r'}^{X} \in \mathbf\Omega_{\infty, \tilde r'}^{X}, \,  \tilde r'\in\tilde{\mathbf{R}}} Sp(\Omega_{\infty, \tilde r'}^{X}) = \bigcup_{\Omega_{\infty, \tilde r}^{X}\in\mathbf{\Omega}_{\infty, \tilde r}^{X}} Sp\left(\Omega_{\infty, \tilde r}^{X}\right) = Sp \left( {}^1\Omega_{\infty, \tilde r}^{X}\right).
\end{equation*}
The last statement, \eqref{e5.15} and \eqref{e5.16*} imply \eqref{e5.21*}.
\end{proof}

Since the Kuratowski limit inferior and limit superior are closed (see, for example, \cite[p.~18]{AF}), we obtain the following corollary of Proposition~\ref{p5.4}.

\begin{corollary}\label{c5.8}
Let $(X, d)$ be an unbounded metric space, $\tilde{r}$ be a scaling sequence. Then the sets
$$
\bigcup_{\Omega_{\infty, \tilde r}^X \in \mathbf{\Omega_{\infty, \tilde r}^X}} Sp(\Omega_{\infty, \tilde r}^X) \quad\text{and}\quad \bigcup_{\Omega_{\infty, \tilde r'}^{X} \in \mathbf{\Omega_{\infty, \tilde r'}^X},\, \tilde r'\in\mathbf{\tilde R}} Sp(\Omega_{\infty, \tilde r'}^{X})
$$
are closed subsets of $[0, \infty)$.
\end{corollary}

Proposition~\ref{pr4.15} claims that every finite, nonempty and strongly rigid metric space $Y$ is isometric to a tangent space $\Omega_{\infty, \tilde r}^{X}$. Using Corollary~\ref{c5.8}, we will show that this is, generally speaking, not so to infinite strongly rigid metric spaces.

\medskip
Let us consider a strongly rigid metric space $(Y, \delta)$ such that:

$(i_1)$ $\delta(x, y)< 2$ for all points $x$, $y\in Y$;

$(i_2)$ $\sup\{\delta(x, y): x, y\in Y\}=2;$

$(i_3)$ The cardinality of the open ball $$B(y^{*}, r)=\{y\in Y: \delta(y, y^{*})<r\}$$ is finite for every $r\in (0, 2)$ and every $y^{*}\in Y.$

\begin{corollary}\label{c5.8*}
Let $(X, d)$ be an unbounded metric space, $\tilde r$ be a scaling sequence, $\Omega_{\infty, \tilde r}^{X}$ be tangent and let $(Y, \delta)$ be a strongly rigid metric space satisfying conditions $(i_1)$-$(i_3).$ If $Y_{1}\subseteq Y$ and $f: \Omega_{\infty, \tilde r}^{X}\to Y_{1}$ is an isometry, then $\Omega_{\infty, \tilde r}^{X}$ is finite. 
\end{corollary}
\begin{proof}
Let $Y_1$ and $f$ satisfy the above conditions and let $y^{*}=f^{-1}(\nu_0),$ $\nu_{0}=\tilde X_{\infty, \tilde r}^{0}$. Conditions $(i_2)$ and $(i_3)$ imply that $Y$ is countable. Consequently $\Omega_{\infty, \tilde r}^{X}$ is separable. Using Corollary~\ref{c5.8}, Corollary~\ref{c5.7} and $(i_2)$ we obtain that $Sp\left(\Omega_{\infty, \tilde r}^{X}\right)$ is a closed subset of $[0, 2]$. Since 
$$
Sp\left(\Omega_{\infty, \tilde r}^{X}\right) = \{\delta(y, y^{*})\colon y\in Y_{1}\}
$$
holds, the set $\{\delta(y, y^{*})\colon y\in Y_{1}\}$ is also closed. If $\Omega_{\infty, \tilde r}^{X}$ is infinite, then $Y_1$ is infinite and, for every sequence $(y_n)_{n\in\mathbb N}$ of distinct points $y_n\in Y_1$, we have 
$$
\lim_{n\to\infty} \delta(y^{*}, y_n) = 2.
$$
Hence
$$
2\in\{d(y, y^{*}): y\in Y_1\}
$$
holds, contrary to $(i_1)$.
\end{proof}

\begin{example}\label{ex5.10}
Let $(Y, \delta)$ be a metric space with $Y=\mathbb N$ and the metric $\delta$ defined such that:
$$\delta(1, 2)=1+\frac{1}{2};$$
$$\delta(1, 3)=1+\frac{2}{3}, \quad \delta(2, 3)=1+\frac{3}{4};$$
$$\delta(1, 4)=1+\frac{4}{5}, \quad \delta(2, 4)=1+\frac{5}{6}, \quad \delta(3, 4)=1+\frac{6}{7};$$
$$\delta(1, 5)=1+\frac{7}{8}, \quad \delta(2, 5)=1+\frac{8}{9}, \quad \delta(3, 5)=1+\frac{9}{10}, \quad \delta(4, 5)=1+\frac{10}{11};$$
$$..................................................................................................................\quad .$$
Then $(Y, \delta)$ is a countable, complete and strongly rigid metric space satisfying conditions $(i_1)$-$(i_3).$ By Corollary~\ref{c5.8*} no tangent space $\Omega_{\infty, \tilde r}^{X}$ is isometric to $(Y, \delta).$
\end{example}

\begin{corollary}\label{c5.9}
Let $(X, d)$ be an unbounded metric space and let $\tilde r$ be a scaling sequence. Then the following statements are equivalent:
\begin{enumerate}
\item[\rm(i)] There is a single-point pretangent space $\Omega_{\infty, \tilde r}^{X}$;
\item[\rm(ii)] All $\Omega_{\infty, \tilde r}^{X}$ are single-point;
\item[\rm(iii)] The equality
\begin{equation*}
\mathop{Li} \limits_{n\to\infty} \left(\frac{1}{r_n} Sp(X)\right) = \{0\}
\end{equation*}
holds;
\item[\rm(iv)] The net $G_{X, \tilde{r}}$ of pretangent spaces to $(X,d)$ at infinity is trivial,
$$
|V(G_{X, \tilde{r}})| = 1.
$$
\end{enumerate}
\end{corollary}

\begin{proof}
It suffices to show that the implication (i) $\Rightarrow$ (ii) is valid. Suppose contrary that there exist pretangent spaces $^1\Omega_{\infty, \tilde r}^{X}$ and $^2\Omega_{\infty, \tilde r}^{X}$ such that
$$
\bigl|\vphantom{1}^1\Omega_{\infty, \tilde r}^{X}\bigr| = 1 \quad \text{and} \quad \bigl|\vphantom{1}^2\Omega_{\infty, \tilde r}^{X}\bigr| \geq 2.
$$
Write $\vphantom{1}^1\tilde{X}_{\infty, \tilde r}$ and $\vphantom{1}^2\tilde{X}_{\infty, \tilde r}$ for the maximal self-stable sets corresponding $^1\Omega_{\infty, \tilde r}^{X}$ and $^2\Omega_{\infty, \tilde r}^{X}$ respectively. Then the equality $\bigl|\vphantom{1}^1\Omega_{\infty, \tilde r}^{X}\bigr| = 1$ implies the equality
$$
\vphantom{1}^1\tilde{X}_{\infty, \tilde r} = \tilde{X}_{\infty, \tilde r}^0.
$$
By statement (v) of Proposition~\ref{pr2.2} we have $\tilde{X}_{\infty, \tilde r}^0 \in \vphantom{1}^2\Omega_{\infty, \tilde r}^{X}$. It follows from the inequality $\bigl|\vphantom{1}^2\Omega_{\infty, \tilde r}^{X}\bigr| \geq 2$ that
$$
\vphantom{1}^2\tilde{X}_{\infty, \tilde r} \setminus \tilde{X}_{\infty, \tilde r}^0 \neq \varnothing.
$$
Consequently we have
$$
\vphantom{1}^1\tilde{X}_{\infty, \tilde r} \subseteq \vphantom{1}^2\tilde{X}_{\infty, \tilde r} \quad \text{and} \quad \vphantom{1}^2\tilde{X}_{\infty, \tilde r} \setminus \vphantom{1}^1\tilde{X}_{\infty, \tilde r} \neq \varnothing.
$$
Since $\vphantom{1}^2\tilde{X}_{\infty, \tilde r}$ is self-stable, the set $\vphantom{1}^1 \tilde{X}_{\infty, \tilde r}$ is not maximal self-stable, contrary to the definition.
\end{proof}

Using Corollary~\ref{c5.9} we can construct an unbounded metric space $(X,d)$ such that there exist single-point pretangent spaces but these spaces are never the tangent spaces to $(X,d)$ at infinity.

\begin{example}\label{ex5.7}
Let $\mathbb{Z}$ be the set of all integer numbers, $t \in (1,\infty)$ and let $X$ be a subset of the real line $\mathbb{R}$ (with the standard metric $d(x,y)=|x-y|$) such that $x \in X$ if and only if $x=0$ or $x = t^i$ for some $i \in \mathbb{Z}$. Let us define a scaling sequence $\tilde{r} = (r_n)_{n \in \mathbb N}$ as
\begin{equation}\label{e5.22}
r_n := t^{n/2}, \quad n \in \mathbb N
\end{equation}
and put $p=0$. Then we have
$$
Sp(X) = \{|x-0|\colon x \in X\} = X
$$
and
\begin{equation}\label{e5.23}
\frac{1}{r_n} Sp(X) = \begin{cases}
X & \text{if $n$ is even}\\
\sqrt{t}X & \text{if $n$ is odd}.
\end{cases}
\end{equation}
It is easy to see that the inclusion
\begin{equation}\label{e5.24}
\mathop{Li} \limits_{n\to\infty} \left(\frac{1}{r_n} Sp(X)\right) \subseteq \mathop{Li} \limits_{k\to\infty} \left(\frac{1}{r_{n_k}} Sp(X)\right)
\end{equation}
holds for every infinite subsequence $(r_{n_k})_{k \in \mathbb N}$ of $\tilde{r}'$. Using \eqref{e5.23} and \eqref{e5.24} with $n_k=2k$, $k \in \mathbb{N}$ and with  $n_k=2k+1$ we obtain
$$
\mathop{Li} \limits_{n\to\infty} \left(\frac{1}{r_n} Sp(X)\right) \subseteq X
$$
and, respectively,
$$
\mathop{Li} \limits_{n\to\infty} \left(\frac{1}{r_n} Sp(X)\right) \subseteq \sqrt{t} X.
$$
Consequently, we have
$$
\mathop{Li} \limits_{n\to\infty} \left(\frac{1}{r_n} Sp(X)\right) \subseteq (\sqrt{t} X) \cap X = \{0\}.
$$
It is clear that
$$
0 \in \mathop{Li} \limits_{n\to\infty} \left(\frac{1}{r_n} Sp(X)\right).
$$
Thus we obtain the equality
$$
\mathop{Li} \limits_{n\to\infty} \left(\frac{1}{r_n} Sp(X)\right) = \{0\}.
$$
Now Corollary~\ref{c5.9} implies that, for $\tilde{r} = (r_n)_{n \in \mathbb N}$ defined by \eqref{e5.22}, there is a unique pretangent space $\Omega_{\infty, \tilde r}^X$ and this space is single-point. A simple calculation shows that the equality
\begin{equation}\label{e5.25}
p^+(Sp(X), \infty) = \frac{t-1}{t}
\end{equation}
holds. Consequently, by Theorem~\ref{th5.3} the metric space $(X,d)$ does not have any single-point \textit{tangent} spaces at infinity.
\end{example}

Letting, at equality \eqref{e5.25}, $t$ to $1$ we obtain the following proposition.

\begin{proposition}\label{p5.8}
For every $\varepsilon >0$ there are an unbounded metric space $(X,d)$ and a scaling sequence $\tilde{r}$ such that $G_{X, \tilde{r}}$ is trivial and
$$
p^+ \bigl(Sp(X), \infty\bigr) < \varepsilon
$$
holds.
\end{proposition}

In the previous proposition, we considered the metric spaces having an arbitrary small positive porosity at infinity. What happens if this porosity becomes zero?

\begin{proposition}\label{p5.11}
Let $(X,d)$ be an unbounded metric space, $p \in X$. If $Sp(X)$ is a nonporous set, then the inequality
\begin{equation}\label{e5.26}
|\Omega_{\infty, \tilde r}^X| \geq 2
\end{equation}
holds for every pretangent space $\Omega_{\infty, \tilde r}^X$.
\end{proposition}
\begin{proof}
Suppose $Sp(X)$ is nonporous at infinity, i.e.,
\begin{equation}\label{e5.27}
p^+(Sp(X), \infty) = 0
\end{equation}
holds. Let $\tilde{r} = (r_n)_{n \in \mathbb N}$ be a scaling sequence. By Corollary~\ref{c5.9} it suffices to show that there is a pretangent space $\Omega_{\infty, \tilde r}^X$ satisfying~\eqref{e5.26}. From Definition~\ref{d5.1} and \eqref{e5.27} it follows that
\begin{equation}\label{e5.28}
\lim_{n\to\infty} \frac{l(\infty, r_n, Sp(X))}{r_n} = 0,
\end{equation}
where $l(\infty, r_n, Sp(X))$ is the length of the longest interval in $[0, r_n)\setminus Sp(X)$. Write
$$
\tau_n := \sup ([0, r_n)\cap Sp(X)), \quad n \in \mathbb{N}.
$$
Then~\eqref{e5.28} implies the equality
$$
\lim_{n\to\infty} \frac{r_n - \tau_n}{r_n} = 0.
$$
Thus
\begin{equation}\label{e5.29}
\lim_{n\to\infty} \frac{\tau_n}{r_n} = 1
\end{equation}
hold. It is easy to see that, for every $n \in \mathbb N$, we have
\begin{equation}\label{e5.30}
\tau_n = \sup\{d(p,x)\colon x \in B(p, r_n)\},
\end{equation}
where $B(p, r_n)$ is the open ball $\{x \in X \colon d(x,p) < r_n\}$. It follows from~\eqref{e5.29}, \eqref{e5.30} and the definition of $Seq(X, \tilde{r})$ that there is $\tilde{x} \in Seq(X, \tilde{r})$ such that
$$
\tilde{\tilde{d}}_{\tilde{r}} (\tilde{x}) = \lim_{n\to\infty} \frac{d(x_n, p)}{r_n} = 1.
$$
Consequently if $\tilde{X}_{\infty, \tilde{r}}$ is maximal self-stable subset of $Seq(X, \tilde{r})$ such that $\tilde{x} \in \tilde{X}_{\infty, \tilde{r}}$, then the inequality
$$
|\Omega_{\infty, \tilde r}^X| \geq 2
$$
holds for the metric identification $\Omega_{\infty, \tilde r}^X$ of $\tilde{X}_{\infty, \tilde{r}}$.
\end{proof}

\section{Acknowledgments}

The research of authors was supported by the grant of the State Fund for Fundamental Research of Ukraine (project F71/20570). The first author was also supported by the National Academy of Sciences of Ukraine in the frame of scientific research project for young scientists ``Geometric properties of metric spaces and mappings in Finsler manifolds''.

\medskip

\textbf{Viktoriia Bilet}

Institute of Applied Mathematics and Mechanics of NASU, Dobrovolskogo Str. 1, Sloviansk 84100, Ukraine

\textbf{E-mail:} viktoriiabilet@gmail.com

\bigskip

\textbf{Oleksiy Dovgoshey}

Institute of Applied Mathematics and Mechanics of NASU, Dobrovolskogo Str. 1, Sloviansk 84100, Ukraine

\textbf{E-mail:} oleksiy.dovgoshey@gmail.com
\end{document}